\documentclass[11pt]{amsart}

\usepackage{amscd}
\usepackage{xypic}  %commu
\usepackage{amssymb}
\usepackage{amsthm}
\usepackage{epsfig}
\usepackage{paralist}
\usepackage{url}
\usepackage[hidelinks=true,hypertexnames=false]{hyperref}
\usepackage{tikz}

\usepackage[all,cmtip]{xy}
\usepackage{amsmath}
\usepackage{color}

\newtheorem{thm}{Theorem}[section]  

\newtheorem{lemma}[thm]{Lemma}
\newtheorem{prop}[thm]{Proposition}
\newtheorem{proposition}[thm]{Proposition}
\newtheorem{claim}[thm]{Claim}

\theoremstyle{definition}

\newtheorem{remark}[thm]{Remark}

\newtheorem{notation}[thm]{Notation }
 \newtheorem{defn}[thm]{Definition}

 \newtheorem{example}[thm]{Example}

\newtheorem*{example1*}{Example: The case $\phi_1=1$}

\newtheorem*{example2*}{Example: The case $\phi_1=2$}

\newtheorem*{example3*}{Example: The case $\phi_1=3$}

\def\min{\operatorname{min}}

\def\max{\operatorname{max}}

\def\c1{\operatorname{c_1}}
\def\c2{\operatorname{c_2}}

\def\Sym{\operatorname{Sym}}

\def\CC{{\mathbb C}}

\def\ZZ{{\mathbb Z}}

\def\PP{{\mathbb P}}

\def\D{{\mathcal D}}
\def\DD{{\mathbb D}}

\def\L{{\mathcal L}}

\def\N{{\mathcal N}}
\def\O{{\mathcal O}}
\def\I{{\mathcal J}}

\def\E{{\mathcal E}}

\def\H{{\mathcal H}}

\def\V{{\mathcal V}}
\def\W{{\mathcal W}}
\def\C{{\mathcal C}}

\def\e{\mathfrak{e}}
\def\s{\mathfrak{s}}
\def\f{\mathfrak{f}}
\def\c{\mathfrak{c}}
\def\d{\mathfrak{d}}

\def\x{\times}                   % product (fiber)
\def\cong{\simeq}

\def\+{\oplus}               % direct sum
\def\*{\otimes}                  % tensor product
\def\Bl{\operatorname{Bl}}

\def\Shext{\operatorname{ \mathfrak{e}\mathfrak{x}\mathfrak{t} }}

\def\Pic{\operatorname{Pic}}

\def\Num{\operatorname{Num}}

\def\Bl{\operatorname{Bl}}

\begin{document}

\title{Nonemptiness of Severi varieties on Enriques surfaces} 

\author[C.~Ciliberto]{Ciro Ciliberto}
\address{Ciro Ciliberto, Dipartimento di Matematica, Universit{\`a} di Roma Tor Vergata, Via \
della Ricerca Scientifica, 00173 Roma, Italy}
\email{cilibert@mat.uniroma2.it}

\author[T.~Dedieu]{Thomas Dedieu}
\address{Thomas Dedieu,                                                                     
Institut de Math{\'e}matiques de Toulouse--UMR5219,                                          
Universit{\'e} de \linebreak Toulouse--CNRS,                                                            
UPS IMT, F-31062 Toulouse Cedex 9, France}
\email{thomas.dedieu@math.univ-toulouse.fr}

\author[C.~Galati]{Concettina Galati}
\address{Concettina Galati, Dipartimento di Matematica e Informatica, Universit{\`a} della \linebreak Calabria, via P. Bucci, cubo 31B, 87036 Arcavacata di Rende (CS), Italy}
\email{concettina.galati@unical.it}

\author[A.~L.~Knutsen]{Andreas Leopold Knutsen}
\address{Andreas Leopold Knutsen, Department of Mathematics, University of Bergen, Postboks 7800,
5020 Bergen, Norway}
\email{andreas.knutsen@math.uib.no}

\date{\today}
%\keywords{}

%\subjclass{}

\begin{abstract}  Let $(S,L)$ be a general polarized Enriques surface,
  with $L$ not numerically 2--divisible. We prove the existence of
  regular components of all Severi varieties of irreducible nodal
  curves in the linear system $|L|$, i.e., for any number of nodes
  $\delta=0, \ldots, p_a(L)-1$.  This solves a classical open problem and gives a positive answer to a recent conjecture of Pandharipande--Schmitt, under the additional condition of non--2--divisibility.
 \end{abstract}

\maketitle

\section{Introduction}

Let $S$ be a smooth, projective complex surface and $L$ a line bundle on $S$. Let $p_a(L)=\frac{1}{2}L \cdot (L+K_S)+1$ denote the arithmetic (or sectional) genus of $L$. For any integer $\delta$ satisfying $0 \leq \delta \leq p_a(L)$ we denote by
$V_{|L|,\delta}(S)$ the {\it Severi variety} parametrizing irreducible $\delta$-nodal curves in $|L|$. A heuristic count shows that the \emph{expected dimension} of $V_{|L|,\delta}(S)$ is $\dim (|L|)-\delta$. 

Severi varieties were introduced by Severi in \cite[Anhang F]{sev}, where he proved  that all Severi varieties
of irreducible $\delta$-nodal curves of degree $d$ in $\PP^2$ are nonempty and smooth of the expected dimension. Severi also claimed irreducibility of such varieties, but his proof contains a gap. The irreducibility was proved by Harris in \cite{Ha}.

Severi varieties on other surfaces have received much attention in recent years, especially in connection with enumerative formulas computing their degrees (see \cite {Be, BOPY, CH, De, Go, KMPS, LS, Tz, YZ}). 
Nonemptiness, smoothness, dimension and irreducibility for  Severi varieties have been widely investigated on various rational surfaces (see, e.g., \cite{GLS, Ta, Te,  Te2, Ty}), as well as K3 and abelian surfaces (see, e.g., \cite{chen, KL, KLM, LS, MM, Ta2, Za}). 
 Extremely little is known on other surfaces. In particular,
Severi varieties may have unexpected behaviour: examples are given in \cite{CC} of surfaces of general type with reducible Severi varieties, and also with components of dimension different from the expected one.

In this paper we consider the case of Enriques surfaces. If $S$ is an Enriques surface, it is known (cf.\ \cite[Prop. 1]{indam}) that
$V_{|L|,\delta}(S)$, {\it if nonempty}, is smooth and every irreducible component has dimension either $p_a(L)-\delta-1$ or $p_a(L)-\delta$. Moreover, if $S$ is general in moduli, the latter case can only occur if $L$ is $2$-divisible in $\Pic (S)$. Any component of dimension
$p_a(L)-\delta-1$ is called {\it regular}, and these components can only be nonempty for
$\delta \leq p_a(L)-1$, that is, they parametrize nodal curves of genus at least one.
The nonemptiness problem has  remained  open until now.

For any integer $g \geq 2$, let 
$\E_{g}$ denote the moduli
space of complex polarized 
Enriques surfaces $(S,L)$ 
of {\it (sectional) genus} $g$, that is, $S$ is an Enriques surface and $L$ is an ample line bundle on $S$ such that
$L^2=2g-2$. Thus, $g$ is the arithmetic genus of all curves in the linear system
$|L|$. The spaces $\E_g$ have many irreducible components. A way to determine these has recently been given in \cite{kn-JMPA}, after partial results were obtained in \cite{cdgk}, cf.\ Theorem \ref{thm:fundcoef} below.

Denote by $\E_g[2]$ the locus in $\E_{g}$ parametrizing pairs $(S,L)$ such that
$L$ is $2$-divisible in $\Num(S)$.
The main result of this paper settles the existence of regular components of all Severi varieties
on general polarized 
Enriques surfaces outside $\E_g[2]$:

\begin{thm} \label{thm:main}
  Let $(S,L)$ be a general element of any irreducible component of $\E_{g} \setminus \E_g[2]$. 
Then $V_{|L|,\delta}(S)$ is nonempty and has a regular component, of dimension $g-1-\delta$, for all $0 \leq \delta<g$.
\end{thm}
 
By \cite[Cor. 1]{indam}, the theorem follows as soon as one proves the case of maximal $\delta$, that is, $\delta=g-1$, in which case the parametrized curves are elliptic. 

We note that Theorem \ref {thm:main} implies  a conjecture due to Pandharipande and Schmitt regarding smooth curves of genus $g\geq 2$ on Enriques surfaces (see \cite[Conj. 5.1] {PS}). Our result implies this conjecture 
for curves whose classes are not 2--divisible (see \cite [Prop. 2.2 and text after Conj. 5.1]{PS}).

We shall prove Theorem \ref {thm:main} by degenerating a general
Enriques surface to the union of two surfaces $R$ and $P$, birational
to the symmetric square of a general elliptic curve and the projective
plane respectively, and glued along a smooth elliptic curve $T$
numerically anticanonical on each surface.
We need the assumption that $L$ is not $2$-divisible to ensure that
the degenerations of the curves we are interested in
do not contain the curve $T = R \cap S$
(see Lemma~\ref{lemma:R3}), which is well known to be a
major issue in the general context of degenerations.

We introduce the degenerations we need in \S \ref {sec:flatlim}. On
such a semi--stable limit we identify suitable curves that deform to
elliptic nodal curves on the general Enriques surface and with the
prescribed linear equivalence class that are rigid,
i.e., they do not move in a positive dimensional family.  As remarked
above, this suffices to prove the theorem.  The aforementioned
suitable curves consist, apart from some $(-1)$-curves as components,
of an irreducible nodal elliptic curve $C_R$ on $R$ and an irreducible
nodal rational curve $C_P$ on $P$ intersecting at one single point on
$T$, where both $C_R$ and $C_P$ are smooth and have a contact of high
order.
Such curves are
members of so-called logarithmic Severi varieties on the surfaces on
which they lie. We develop all necessary tools and results on such
varieties on the two types of surfaces in question in \S \ref
{sec:GSV}.

The analysis of the conditions under which the limit
curves actually deform to rigid nodal elliptic curves on the general
Enriques surface is performed in the crucial \S \ref
{sec:prepar}. This includes the most delicate part of our proof
(Propositions \ref{prop:metodo1} and \ref{prop:metodo2}), which
consists in showing that the above mentioned curves $C_R$ and $C_P$
can be put together nicely.
We end up eventually with numerical conditions to be
verified by the line bundles determined on each component of the limit
surfaces.

An important ingredient next is the description of all
components of moduli spaces of polarized Enriques surfaces in terms of
decompositions of the polarizing line bundles into effective isotropic
divisors as developed recently in \cite {cdgk, kn-JMPA}, which we
review in \S \ref {sec:iso}. The corresponding identification of
suitable isotropic Cartier divisors on the limit surfaces is done in
\S \ref {sec:speciso}. Finally, \S \ref{sec:proofmain} is devoted to
exhibiting, for each component of the moduli spaces of polarized
Enriques surfaces, a suitable isotropic decomposition of the limit
polarising line bundle, such that its restriction on each component
verifies the conditions  necessary to deform  the curves mentioned above.

\vspace{0.3cm} 
\noindent
{\it Acknowledgements.} The authors would like to thank Johannes Schmitt for useful correspondence and two referees for useful remarks. The authors acknowledge funding
from MIUR Excellence Department Project CUP E83C180 00100006 (CC),
project FOSICAV within the  EU  Horizon
2020 research and innovation programme under the Marie
Sk{\l}odowska-Curie grant agreement n.  652782 (CC, ThD),
 GNSAGA of INDAM (CC, CG), the Trond Mohn 
Foundation Project ``Pure Mathematics in Norway'' (ALK, CG) and grant 261756 of the Research Council of Norway (ALK).

\section{Flat limits of Enriques surfaces} \label{sec:flatlim}

In this section we will introduce the semi--stable degenerations of general Enriques surfaces that we will use in our proof of Theorem \ref {thm:main}. 

Let $E$ be a smooth elliptic curve. Denote by $\+$ (and $\ominus$) the group operation on $E$  and by $e_0$ the neutral element. Let $R:=\Sym^2(E)$ and
$\pi: R \to E$ be the (Albanese) projection map sending $x+y$ to $x\+ y$.
We denote the fiber of $\pi$ over a point $e \in E$ by 
\[ \f_e:=\pi^{-1}(e)=\{ x+y \in \Sym^2(E) \; | \; x\+ y=e \; \mbox{(equivalently,} \; 
x + y \sim e+e_0)\},\]
which is the $\PP^1$ defined by the linear system $|e+e_0|$. (Here, and throughout the paper, $\sim$ denotes linear equivalence of divisors.) We denote the algebraic equivalence class of the fibers by $\f$. 
Symmetric products of elliptic curves have been studied in detail in
\cite{CaCi}, to which we will frequently refer in this paragraph.

For each $e \in E$, we define the curve $\s_e$ (called $D_e$ in \cite{CaCi})
as the image of the section $E \to R$ mapping $x$ to $e+ (x \ominus e)$.
We let $\s$ denote the algebraic equivalence class of these sections, which are the ones with minimal self-intersection, namely $1$, cf.\ \cite{CaCi}. We note that $\Sym^2(E)$ is the $\mathbb P^1$--bundle on $E$ with invariant $-1$.  We observe for later use that
 for $x \neq y$ we have
\begin{equation} \label{eq:duesez} 
\s_x \cap \s_y=\{ x+y\}.
\end{equation}
We also note that 
\begin{equation*} \label{eq:can} 
K_R \sim -2\s_{e_0}+\f_{e_0}.
\end{equation*}

Let $\eta$ be any of the three nonzero $2$-torsion points of $E$. The map $E \to R$ defined by  $e\mapsto e + (e \+ \eta)$ realizes $E$ as an unramified double cover of its image curve
$T:= \{ e+ (e\+\eta) \; | \; e \in E\}$,
which is a smooth elliptic curve. We have
\begin{equation*} \label{eq:T} 
T \sim -K_R+\f_{\eta}-\f_{e_0} \sim 2\s_{e_0}-2\f_{e_0}+\f_{\eta},
\end{equation*}
by \cite[(2.10)]{CaCi}. In particular,
$T \not \sim -K_R \; \; \mbox{and} \; \; 2T \sim -2K_R$.

Embed $T$ as a cubic in $P:=\PP^2$. Consider nine (possibly coinciding) points $y_1,\ldots,y_9 \in T$.
Divide the nine points in two subsets, say of $i$ and $9-i$ points, with $0 \leq i \leq 9$. Let $\widetilde{R} \to R$ and $\widetilde{P} \to P$, respectively, denote the blow--ups at the schemes on $T$ determined by these two subsets of $i$ and $9-i$ points, respectively. Denote by $\e_R$ and $\e_P$ the total exceptional divisors on $\widetilde{R}$ and $\widetilde{P}$, respectively, and 
denote the strict transforms of $T,\s,\f$ with the same symbols. We have 

\begin{eqnarray} \label{eq:T3} 
  T  \sim   2\s_{e_0}-2\f_{e_0}+\f_{\eta}-\e_R \not \sim -K_{\widetilde{R}}\sim 2\s_{e_0}-\f_{e_0}-\e_R&  \mbox{on} \; \; \widetilde{R},\\
 \label{eq:T33} 2T   \sim -2K_{\widetilde{R}} & \mbox{on} \; \; \widetilde{R}, \\
 \label{eq:TP} T  \sim   3\ell-\e_P  \sim -K_{\widetilde{P}}  & \mbox{on} \; \; \widetilde{P}
\end{eqnarray}
where $\ell$ is the pull--back on  $\widetilde{P}$ of a general line in $P$. 
Define
$X=\widetilde{R} \cup_T \widetilde{P}$ as the surface obtained by  gluing
$\widetilde{R}$ and $\widetilde{P}$ along $T$. 
Denote by $\D_{[i]}$ the family of such surfaces.
It is easy to see that $\D_{[i]}$ is irreducible of dimension $10$ (when one also allows $E$ to vary in moduli). We define
$\D:=\cup_{i=0}^9 \D_{[i]}$.

Let $X$ be a member of $\D$.
The {\it first cotangent sheaf} $T^1_{X}:=\Shext^1_{\O_X}(\Omega_{X},\O_{X})$ of $X$ (cf.\ \cite[Cor. 1.1.11]{ser} or \cite[\S 2]{fri}) 
satisfies
\begin{equation*} \label{eq:ss}
 T^1_{X} \cong \N_{T/\widetilde{R}} \* \N_{T/\widetilde{P}} 
\end{equation*}
by \cite[Prop. 2.3]{fri}, which is trivial if and only if the nine points satisfy the condition
\begin{equation}
  \label{eq:cond}
  y_1+\cdots+y_9 \in |\N_{T/R} \* \N_{T/P}|.
\end{equation}
  Thus, $X$ is {\it semi-stable} if and only if \eqref{eq:cond} holds, cf.\ \cite[Def. (1.13)]{fri} and \cite[(0.4)]{fri2}. We will denote by $\D^*_{[i]}$ the subfamily of $\D_{[i]}$ consisting of semi--stable surfaces.
It is easy to see that $\D^*_{[i]}$ is irreducible of dimension $9$.
We define
$\D^*:=\cup_{i=0}^9 \D^*_{[i]}$.

We recall that a Cartier divisor, or a line bundle, $\L$ in $\Pic (X)$, is a pair
$(L',L'')$ such that $[L'] \in \Pic (\widetilde{R})$, $[L''] \in \Pic( \widetilde{P})$ and $L'|_T \cong
L''|_T$. Since $T$ is numerically equivalent to the anticanonical divisor on both $\widetilde{R}$ and $\widetilde{P}$, we have
\begin{equation*} \label{eq:even}
 \L^2=(L')^2+(L'')^2=2p_a(L')-2+2p_a(L'')-2+2d, \; \; d:=L'\cdot T=L''\cdot T.
\end{equation*}

The canonical divisor $K_{X}$ is represented by 
\begin{equation*}
  \label{eq:canrist}
  K_X=(K_{\widetilde{R}}+T,0)=(\f_{\eta}-\f_{e_0},0) \; \; \mbox{in} \; \; \Pic (\widetilde{R}) \x \Pic (\widetilde{P}). 
\end{equation*}
 In particular, by \eqref{eq:T3}-\eqref{eq:TP} we have
\begin{equation}
  \label{eq:doppioz}
  K_X \not\sim 0 \; \; \mbox{and} \; \; 2K_X\sim 0.
\end{equation}
By \cite[Lemma 3.5]{kn-JMPA} the Cartier divisor $K_X$ is the only nonzero torsion element of $\Pic (X)$. (The proof is for $i=2$, but carries over to the general case.)

\begin{remark}\label{rem:ingcro10}  There are exactly two elements of $\Pic^0(\widetilde R)\cong E$ that restrict trivially on $T$, namely $\mathcal O_{\widetilde R}$ and $\mathcal O_{\widetilde R}(\f_\eta-\f_{e_0})$ (see \cite [Lemma 3.3]{kn-JMPA}). Accordingly, for any $[L'] \in \Pic (\widetilde{R})$ and $[L''] \in \Pic(\widetilde{P})$ such that $L'\cdot T=L''\cdot T$, there are two  line bundles $\overline L'$ on $\widetilde{R}$ numerically equivalent to $L'$ such that 
$(\overline L', L'')$ is a line bundle on $X$. By \eqref {eq:doppioz}, their difference is $K_X$.  These line bundles are numerically equivalent and we will denote by $[L',L'']$ their numerical equivalence class. 
\end{remark}

By \eqref {eq:cond}, if $X$ is semi--stable it also carries the Cartier divisor $\xi$ represented by the pair
\begin{equation} \label{eq:xi}
  \xi= (T,-T)\sim (2\s_{e_0}-2\f_{e_0}+\f_{\eta}-\e_R,-3\ell+\e_P)
\end{equation}
in $\Pic (\widetilde{R}) \x \Pic (\widetilde{P})$ (see \cite[(3.3)]{fri2}).

The central result for our purposes is:

\begin{thm} \label{thm:deform}
  Let $y_1,\ldots,y_9 \in T$ be general such that  $X=\widetilde{R} \cup_T \widetilde{P}$ is a member of $\D^*$. 

    There is a flat family $\pi:\mathfrak{X} \to \DD$ over the unit disc such that $\mathfrak{X}$ is smooth and, setting $S_t:=\pi^{-1}(t)$, we have that
    \begin{itemize}
    \item $S_0=X$, and
     \item $S_t$ is a smooth general Enriques surface for $t\neq 0$. 
    \end{itemize}

Furthermore, denoting by $\iota_t: S_t \subset \mathfrak{X}$  the inclusion, there is a short exact sequence
\[
  \xymatrix{
      0 \ar[r] & \ZZ\cdot\xi \ar[r] & \Pic (X) \cong H^2(\mathfrak{X},\ZZ) \ar[r]^{\iota_t^*} & H^2(S_t,\ZZ) \cong \Pic (S_t)  \ar[r] & 0.}
  \]
\end{thm}

\begin{proof}
  This follows from \cite[Prop. 3.7, Thm. 3.10 and Cor. 3.11]{kn-JMPA} in the case where $X$ lies in $\D_{[2]}^*$. Once we have the statement in this case, we can prove it in the other cases by making a birational transformation of $\mathfrak{X}$ to flop any of the exceptional curves between $\widetilde{P}$ and $\widetilde{R}$ (see for example \cite [\S 4.1]{cm}, where the flop is called a {\it 1--throw}).  
\end{proof}

\section{Logarithmic Severi varieties} \label{sec:GSV}

Theorem \ref {thm:main} will be proved by degenerating a general Enriques surface to a surface $\widetilde R\cup_T  \widetilde{P}$ in $\D^*$. It will be essential to construct curves on $\widetilde R\cup_T  \widetilde{P}$ that will deform to nodal irreducible elliptic curves on the general Enriques surface. As we will see in \S \ref{sec:prepar}, the good limit curves on  $\widetilde R$ and $  \widetilde{P}$ are nodal curves with high order tangency with $T$ at the same point on each component. These are members of so-called {\it logarithmic Severi varieties}, parametrizing nodal curves with given tangency conditions to a fixed curve. This will be the topic of this section.
We start with some general definitions and results:

  \begin{defn} \label{def:GSV}
    Let $S$ be a smooth projective surface, $T \subset S$ a smooth, irreducible curve and  
    $L$ a line bundle or a divisor class on $S$.
Let $g$ be an integer satisfying $0 \leq g \leq p_a(L)$.

For any effective divisor $\d=m_1p_1+\cdots+m_lp_l$ on $T$, where the
$p_i$ are pairwise distinct, we denote by $V_{g,\d}(S,T, L)$ the locus of curves in $S$ such that
    \begin{itemize}
    \item $C$ is irreducible of geometric genus $g$ and algebraically equivalent to $L$,
    \item denoting by $\nu:\widetilde{C} \to S$ the normalization of $C$ composed with the inclusion $C \subset S$, there exists $q_i \in \nu^{-1}(p_i)$ such that $\nu^*T$ contains $m_iq_i$, for all $i \in \{1,\ldots,l\}$.
    \end{itemize}

    For any integer $m$ satisfying $0<m \leq L \cdot T$ we let
    $V_{g,m}(S,T, L)$ denote the locus of curves contained in some
$V_{g,mp}(S,T, L)$ for some (non-fixed) $p \in T$.

We denote by $V_{g,m}^*(S,T, L)$ the open sublocus of $V_{g,m}(S,T, L)$
parametrizing curves that are {\it smooth} at the intersection points with $T$ and otherwise {\it nodal}.
\end{defn}
  
In the sequel $\equiv$ will denote numerical equivalence of divisors.  We will need: 

\begin{prop} \label{prop:TD}
  Let $S,T,L,\d,g$ and $m$ be as in Definition \ref{def:GSV}. Assume that $T \equiv -K_S$.
  \begin{itemize}
  \item[(i)] If  $L \cdot T > \sum_{i=1}^l m_i$,
    % and $V_{g,\d}(S,T, L) \neq \emptyset$,
    then all irreducible components of $V_{g,\d}(S,T, L)$ have
    dimension $g-1+L \cdot T - \sum_{i=1}^l m_i$.
  \item[(ii)] %If $V_{g,m}(S,T, L) \neq \emptyset$, then
    All irreducible components of $V_{g,m}(S,T, L)$ have dimension
    $g+L \cdot T-m$.
    \item[(iii)] If $m \leq L \cdot T-2$, then the general member $[C]$
      in any component of $V_{g,m}(S,T, L)$ is smooth at its intersection
      points with $T$;
      moreover, if we fix $G\subset S$ any curve not having $T$ as an
      irreducible component, and $\Gamma \subset S$ any finite set,
      then for general $[C]$,
      the curve $C$ is tranverse to $G$ and does not intersect
      $\Gamma$. 
      \item[(iv)] If $m \leq L \cdot T-3$, then the general member in any component of $V_{g,m}(S,T, L)$ is nodal.
 \end{itemize} 
\end{prop}

\begin{proof}
  The result follows from  \cite[\S 2]{CH}, as outlined in \cite[Thm. (1.4)]{noteTD}. 
\end{proof}

\subsection{Families of blown-up surfaces}
\label{s:logS-family}

\label{cazztost} We will also need to work in families in the following way.  For $S=R$ or $P$ containing $T$ as above and for any nonnegative integer $n$ 
we consider the family
$\mathcal{S}^{\langle n\rangle} \to T^n$ with fiber
$\Bl_{y_1,\ldots,y_n}(S)$ over $(y_1,\ldots,y_n) \in T^n$, the blow--up of $S$ at $y_1,\ldots,y_n$ (when the points are coinciding, this has to be interpreted as blowing up curvilinear schemes on $T$). To be precise, the fibers are {\it marked}, in the sense that their (total) exceptional divisors are labelled with $1,\ldots,n$.  Whenever we have a line bundle on a single surface $\Bl_{y_1,\ldots,y_n}(S)$, we can write it in terms of the generators of $\Pic (S)$ and of the exceptional divisors over each $y_i$, and thus we can extend it to a relative line bundle on the whole family $\mathcal{S}^{\langle n\rangle}$ in the obvious way. We will therefore mostly not distinguish notationally between a relative line bundle $L$ and its restriction to any surface in the family.

Similarly, there is for all $i=1,\ldots,n$ a relative (total) exceptional divisor $\e_i$ on $\mathcal{S}^{\langle n\rangle}$, whose fiber over a point $(y_1,\ldots,y_n) \in T^n$ is the exceptional divisor
on $\Bl_{y_1,\ldots,y_n}(S)$ over the point $y_i$, which we by abuse of notation still denote by $\e_i$. 

\begin{defn} \label{def:posi}
  Let $L$ be a relative line bundle on $\mathcal{S}^{\langle n\rangle}$. The {\it value of $L$ on the $i$th exceptional divisor} is the number $L \cdot \e_i$ on any fiber $\Bl_{y_1,\ldots,y_n}(S)$.
We say that $L$ is {\it positive on the $i$th exceptional divisor} if
 $L \cdot \e_i>0$. 
\end{defn}

We shall consider the relative Hilbert scheme
\[ \H^{\langle n\rangle}_{S,L} \longrightarrow T^n \]
whose fibers are the Hilbert schemes of curves on $\Bl_{y_1,\ldots,y_n}(S)$ algebraically (or equivalently numerically)  equivalent to $L$. 
We have a (possibly empty) scheme
\[ \V^{\langle n\rangle}_{g,m}\left(S,T,L\right) \longrightarrow T^n \]
whose fibers are $V_{g,m}(\Bl_{y_1,\ldots,y_n}(S),T,L) \subset \H^{\langle n\rangle}_{S,L}$ (here, as usual, we denote by $T$ its strict transform on the blow--up).
Taking the closure in $\H^{\langle n\rangle}_{S,L}$, we obtain a (possibly empty) scheme with a morphism
\begin{equation} \label{eq:defnu}
 \nu^{\langle n\rangle}_{g,m}\left(S,T,L\right): \overline{\V^{\langle n\rangle}_{g,m}\left(S,T, L\right)} \longrightarrow T^n ,
  \end{equation}
whose fibers we denote by
\[
 (\nu^{\langle n\rangle}_{g,m})^{-1}(y_1,\ldots,y_n):=
  \overline{V}^{\langle n\rangle}_{g,m}\left(\Bl_{y_1,\ldots,y_n}(S),T,L\right).
\]
Note that for any $(y_1,\ldots,y_n) \in T^n$ one has
\[
 \overline{V_{g,m}(\Bl_{y_1,\ldots,y_n}(S),T,L)}\subseteq \overline{V}^{\langle n\rangle}_{g,m}\left(\Bl_{y_1,\ldots,y_n}(S),T,L\right).\]

\subsection{Logarithmic Severi varieties on blow--ups of the symmetric square of an elliptic curve} \label{sec:GSVR}

Let $T \subset R=\Sym^2(E)$ as defined in \S \ref{sec:flatlim}. 
Let $y_1,\ldots,y_n \in T$ and let $\widetilde R:=\Bl_{y_1,\ldots,y_n}(R)$ denote the blow--up of $R$ at $y_1,\ldots,y_n$, with (total) exceptional divisors $\e_i$ over $y_i$. We denote the strict transforms of $\s$, $\f$ and $T$ on $\widetilde R$ by the same names. We also still denote by $\pi:\widetilde{R} \to E$ the composition of the blow--up $\widetilde{R} \to R$ with the Albanese morphism $R \to E$ (cf.\ beginning of
\S \ref{sec:flatlim}). 
By \eqref{eq:T3}--\eqref {eq:T33} we have
\[ T \equiv -K_{\widetilde R} \equiv 2\s-\f-\e_1-\cdots-\e_n.\]

\begin{defn} \label{def:odd}
   A line bundle or Cartier divisor $L$ on $\widetilde{R}$ is {\it odd} if $L \cdot \f$ is odd.
\end{defn}

\begin{notation}\label{def:symtnm} We denote by
$ \Sym^n(T)_m \subset \Sym^n(T)$
the subscheme consisting of divisors with a point of multiplicity $\geq m$.
\end{notation}

\begin{lemma} \label{lemma:R3}
  Let $L$ be an odd line bundle or Cartier divisor on $\widetilde{R}$. 
Let $m$ be any integer satisfying $1 \leq m \leq L \cdot T$. Then the following hold:

 \begin{itemize} 
 \item [(i)] No curve $C$ in $\overline{V}_{1,m}(\widetilde{R},T ,L)$ contains $T$.

\item [(ii)] For any component $V \subset \overline{V}_{1,m}(\widetilde{R},T, L)$ the restriction map
\begin{eqnarray*}
  V & \longrightarrow & \Sym^{L \cdot T}(T)_m \\
  C &\mapsto & C \cap T
\end{eqnarray*}
is well-defined, finite and surjective. In particular,
$$
\dim(V)=L\cdot T-m+1.
$$
\item[(iii)]  For a general curve $C$ in any component of
$V^*_{1,m}(\widetilde{R},T ,L)$, let $N$ be the reduced subscheme of
$\widetilde R$ supported at the nodes of $C$, and $Z$ any subscheme of $C \cap T$ of degree $C \cdot T-1$. Then the linear system $|\O_{\widetilde{R}}(C) \* \I_{N \cup Z}|$ consists only of $C$.
  
\end{itemize}
\end{lemma}

\begin{proof}
  Assume that we have $C=hT+C'$ in $\overline{V}_{1,m}(\widetilde{R},T, L)$ for some $h>0$, with $C'$ not containing $T$. 
  We have $L \cdot \f=C \cdot \f=2h+C' \cdot \f$, whence $C' \cdot \f>0$ since $L$ is assumed to be odd. Hence $C'$ has at least one component dominating $E$ via $\pi: \widetilde{R} \to E$, and therefore
  $C$ cannot be a limit of an elliptic curve. Thus (i) follows.

It also follows that the restriction map in (ii) is everywhere defined.
The fiber over a $Z \in \Sym^{L \cdot T}(T)$ consists of all curves $C$ in $V$
such that $C \cap T=Z$. This must be finite, for otherwise we would
find a member of the fiber passing through an additional general point $p \in T$, a contradiction (using again that no curve in $V$ contains $T$). Hence the restriction morphism  in (ii) is finite. 
We have $\dim (V)\geq L \cdot T+1-m$ by Proposition \ref{prop:TD}(ii) and semicontinuity, which equals $\dim \left(\Sym^{L \cdot T}(T)_m\right)$. The morphism is therefore surjective and equality holds for the dimension.  This proves (ii).

Let now $C$ be a curve in $V^*_{1,m}(\widetilde{R},T ,L)$ and $Z$ be any subscheme of $C \cap T$ of degree $C \cdot T-1$. Let $\widehat{R} \to \widetilde{R}$ denote the blow--up of $\widetilde{R}$ along $Z$, considered as a curvilinear subscheme of $T$, and let $\widehat{C}$ and $\widehat{T}$ denote the strict transforms of $C$ and $T$, respectively, and $\widehat{L}:=\O_{\widehat{R}}(\widehat{C})$. Then $\widehat{C}$ is a member of $V^*_{1,1}(\widehat{R},\widehat{T},\widehat{L})$. To prove (iii)  we may reduce to proving that if $X$ is a general member of a component of $V^*_{1,1}(\widehat{R},\widehat{T},\widehat{L})$, and $N$ is the subscheme of its nodes, then the linear system $|\O_{\widehat{R}}(X) \* \I_N|$ consists only of $X$. 

Let $\delta=p_a(C)-1$. The variety $V^*_{1,1}(\widehat{R},\widehat{T},\widehat{L})$ is the open subset of the {\it Severi variety of $\delta$-nodal curves algebraically equivalent to $\widehat{L}$} consisting of curves with nodes off $\widehat{T}$. All of its components have dimension $\widehat{L} \cdot \widehat{T}=1$ by (ii) (or Proposition \ref{prop:TD}(ii)), and it is smooth by standard arguments (see, e.g., \cite[Prop. 2.2]{cdgk2}). Let $W$ be any component of $V^*_{1,1}(\widehat{R},\widehat{T},\widehat{L})$. Then $W$ is fibered over $\Pic^0(E) \cong E$ in subvarieties $W_{\widehat{L}'}$ parametrizing $\delta$-nodal curves in $|\widehat{L}'|$, where $\widehat{L}'$ is any line bundle numerically equivalent to $\widehat{L}$. By (ii) the linear equivalence classes of the curves in $W$ vary. Thus $W_{\widehat{L}'}$ is nonempty for general $\widehat{L}'$, whence  smooth and zero-dimensional. The tangent space to $W_{\widehat{L}'}$ at any point $[X]$ is isomorphic to $H^0(\widehat{L}' \* \I_N)/\CC$, where $N$ is the scheme of nodes of $X$ (see, e.g., \cite[\S 1]{CS}). In particular, for
a general $X$ in $W$ we have
 \[
 \dim \left(|\O_{\widehat{R}}(X) \* \I_N|\right)=
 h^0(\O_{\widehat{R}}(X) \* \I_N) -1  =  \dim (W_{\widehat{L}'})=0,
\]
whence $|\O_{\widehat{R}}(X) \* \I_N|$ consists only of
 $X$, as desired. This proves (iii).
\end{proof}

In view of part (iii) of the previous result, we  introduce  the following:

\begin{notation} \label{def:dustelle}
  We let $V^{**}_{1,m}(\widetilde R,T, L)$ denote the open subvariety of $V^{*}_{1,m}(\widetilde R,T, L)$ para- \linebreak metrizing  curves $C$ such that, for $N$ its scheme of nodes and for every subscheme $Z$ of $C \cap T$ of degree $C \cdot T-1$, the linear system $|\O_{\widetilde{R}}(C) \* \I_{N \cup Z}|$ consists only of $C$.
\end{notation}

The main existence result of this subsection is Proposition \ref{prop:R2} right below. To state it we need a definition:

\begin{defn} \label{def:good}
  A line bundle or Cartier divisor $L$ on $\widetilde R$ is said to verify condition $(\star)$ if it is of the form $L\equiv \alpha\s+\beta \f-\sum_{i=1}^n \gamma_i \e_{i}$  such that:\\
  \begin{inparaenum}
\item [(i)] $\alpha\geq 1$ and $\beta\geq 0$;\\
 \item [(ii)] $\alpha\geq \gamma_i$ for $i=1,\ldots, n$;\\
\item [(iii)] $\alpha+\beta\geq \sum_{i=1}^n\gamma_i$;\\ 
\item [(iv)] $\alpha+2\beta\geq \sum_{i=1}^n\gamma_i+4$ (equivalently, $-L\cdot K_{\widetilde R}\geq 4$).
\end{inparaenum}
\end{defn}

\begin{prop} \label{prop:R2}
  Let $E$ and $y_1,\ldots,y_n \in T$ be general. Assume that $L$ is a line bundle on $\widetilde R$ that is  odd (cf.\ Definition \ref{def:odd}) and satisfies condition $(\star)$ (cf.\ Definition \ref{def:good}).  Then, if $ 0< m \leq  L \cdot T-3$,
the variety
$V^{**}_{1,m}(\widetilde R,T, L)$ (cf.\  Definitions  \ref{def:GSV}
and \ref{def:dustelle}) has pure dimension $L\cdot T-m+1$.
Moreover, for all  curves 	  $G \subset \widetilde R$ not having $T$ as an
irreducible component, the general member of
$V^{**}_{1,m}(\widetilde R,T, L)$ intersects $G$ transversely.
\end{prop}

\begin{proof} By Proposition \ref{prop:TD}(ii)--(iv)  and Lemma \ref{lemma:R3}(iii)  we only need to  prove  non--emptiness of $V_{1,m}(\widetilde R,T,L)$.  Following an idea in the proof of \cite [Thm. 3.10]{CGL}, we will prove this by induction on $m$.
The base case $m=1$ follows from \cite[Prop. 2.3]{cdgk2},

which requires all of (i)--(iv) from condition $(\star)$.

Assume that we have proved non--emptiness of $V_{1,m}(\widetilde R,T,L)$ for some $1 \leq m \leq  L \cdot T-4$.  By Lemma \ref{lemma:R3}(ii) its general member $C$ satisfies
  \[ C \cap T=mp_0+p_1+\cdots+p_l+p_{l+1}, \; \; l=L \cdot T-m-1 \geq 3,\]
  where $p_0,\ldots,p_{l+1}$ are pairwise distinct, general points on $T$. 
  Set $\d=mp_0+p_1+\cdots+p_{l}$. Then $V_{1,\d}(\widetilde R,T, L) \neq \emptyset$ and all its components are one-dimensional, by Proposition \ref{prop:TD}(i). The general member in any component intersects $T$ in $mp_0+p_1+\cdots+p_{l}+q$, where the point $q$ varies in the family,  by Proposition \ref{prop:TD}(iii).  Pick a component $\overline{V}$ of its closure inside the component of the Hilbert scheme of $\widetilde R$ containing $|L|$. After a finite base change, we find a smooth projective curve $B$, a surjective morphism $B \to \overline{V}$ and a family
  \[ \xymatrix{ \C \ar[r]^{f}  \ar[d]_{g} & \widetilde R \\
      B} \]
  of stable maps of genus one such that, setting $\C_b:=g^*b$ for any $b \in B$, the curve
$f_*\C_b$ is a member of $\overline{V}$, and such that 
  \[ f^*T=mP_0+P_1+\cdots+P_l+Q+W,\]
  where
  \begin{itemize}
  \item[(I)] $P_i$ and $Q$ are sections of $g$, for $i \in \{0,\ldots,l\}$,
  \item[(II)] $f(P_i)=p_i$, for $i \in \{0,\ldots,l\}$,
\item[(III)] $f(Q)=T$,
\item[(IV)] $g_*W=0$,
  \item[(V)] $f_*W=0$;
  \end{itemize}
  the latter property follows from the fact that no member of the
  family contains $T$, by Lemma \ref{lemma:R3}(i).

Property (III) implies that $f^{-1}(T)$ is connected as follows.
Consider the Stein factorization
$\mathcal{C} \xrightarrow {f'} R' \xrightarrow h \tilde R$ of $f$.
Then $h^{-1}(T)$ is of pure dimension $1$. Since all irreducible
components of $f^*T$ except $Q$ are contracted by $f$, it follows that
$h^{-1}(T) = f'(Q)$, in particular it is irreducible.
Eventually, since $f'$
has connected fibers,
$f^{-1}(T) = (f')^{-1} (h^{-1}(T))$ is connected.

In particular, $P_0$ and $Q$ are connected by an effective  (possibly zero)  divisor $W' \subset f^{-1}(p_0) \cap \C_{b_0} \subset W$ for some $b_0 \in B$. Thus, 
  \begin{equation}\label{eq:pipi} f_* \C_{b_0} \cap T=(m+1)p_0+p_1+\cdots+p_l, \; \; l \geq 3.\end{equation} 
By the generality of the points $p_0,\ldots,p_l$, they cannot be
contained in any $(-1)$-curve on $\widetilde R$, nor can any two of
them lie in a fiber of $\pi:\widetilde R \to E$. Consequently, $ f_*
\C_{b_0}$ cannot contain any rational component.
Moreover, $f_* \C_{b_0}$ must be a reduced curve by \eqref
{eq:pipi}. Therefore $f_* \C_{b_0}=C$ is an irreducible
curve  of geometric genus one,
hence $\C_{b_0}$ consists of one smooth elliptic curve $\widetilde{C}$
such that $f(\widetilde{C})=C$ and otherwise chains of rational curves
contracted by $f$ and attached to $\widetilde{C}$ at one single point
each. 
Therefore $f^{-1}(p_0) \cap \widetilde{C}$ is a single (smooth) point
of $\widetilde{C}$, hence
$[C] \in V_{1,(m+1)p_0+p_1+\cdots+p_{l}}(\widetilde R,T, L)$
by \eqref{eq:pipi}, which implies
$[C] \in V_{1,m+1}(\widetilde R,T,L)$. 
\end{proof}

\subsection{Logarithmic Severi varieties on blown up planes} \label{sec:GSVP}

Fix a smooth cubic curve $T \subset P=\PP^2$.
Let $y_1,\ldots,y_n \in T$, for $n\geq 0$, and consider the blow--up $\widetilde P:=\Bl_{y_1,\ldots,y_n}(P) \to P$ at $y_1,\ldots,y_n$. We denote the strict transforms of the general line on $P$ by $\ell$ and by $\e_i$ the (total) exceptional divisor over $y_i$.  We denote still by $T$ the strict transform of $T$. Note that $T \sim -K_{\widetilde P} \sim 3\ell-\e_1-\cdots-\e_n$.

The next result parallels Lemma \ref{lemma:R3}.

\begin{lemma} \label{lemma:P3}
Let $L$ be a line bundle or Cartier divisor on $\widetilde{P}$. 
Let $m$ be any integer satisfying $1 \leq m \leq L \cdot T$. Then the following hold:

(i) No curve $C$ in $\overline{V}_{0,m}(\widetilde{P},T, L)$ contains $T$.

(ii) For any component $V \subset \overline{V}_{0,m}(\widetilde{P},T, L)$ the restriction map
\begin{eqnarray*}
  V & \longrightarrow & \Sym^{L \cdot T}(T)_m \\
  C &\mapsto & C \cap T
\end{eqnarray*}
is well-defined and finite, with image $|L \otimes \O_T| \cap \Sym^{L \cdot T}(T)_m$, which has  codimension one. In particular,
$$
\dim(V)=L\cdot T-m.
$$
\end{lemma}

\begin{proof}
  Since the members of $\overline{V}_{0,m}(\widetilde{P},T,L)$ are limits of rational curves, none of them can contain $T$ as a component, which proves (i). As in the proof of Lemma \ref{lemma:R3}(ii), the restriction map is everywhere defined and finite. Its image lies in $|L|_T|\cap \Sym^{L \cdot T}(T)_m$. Since, by Proposition \ref{prop:TD}(ii) and semicontinuity,
  $\dim (V)\geq L \cdot T-m$, which equals $\dim \left(|L|_T| \cap \Sym^{L \cdot T}(T)_m\right)$, the latter is in fact the image. This proves (ii).
\end{proof}

The next result is about the relative version
$\nu^{\langle n\rangle}_{g,m}: \overline{\V^{\langle n\rangle}_{g,m}\left(P,T, L\right)} \longrightarrow
T^n$
of the logarithmic Severi variety
$\overline{V}_{0,m}(\widetilde{P},T, L)$ considered in
Lemma~\ref{lemma:P3} above
(see subsection \ref{s:logS-family}).

\begin{lemma} \label{lemma:P3-family}
(i) Assume that $n>0$ and $L$ is a relative line bundle that is positive on the $i$-th exceptional divisor. Fix a
point $(y_1,\ldots,y_{i-1},y_{i+1}, \ldots,y_n) \in T^{n-1}$. Let $\V$ be any component of 
\[ \{(\nu^{\langle n\rangle}_{g,m})^{-1} (y_1,\ldots,y_{i-1},p,y_{i+1}, \ldots,y_n),\ p \in T\}. \]
Then the restriction map
\[\V \longrightarrow \Sym^{L \cdot T}(T)_m \]
is finite and surjective.

(ii) Assume furthermore that $n \geq 2$ and $L$ is positive, with two different values, on the $i$-th and $j$-th exceptional divisor, $i < j$. Fix any
linear series  $\mathfrak{g}$ of type $g^1_2$ on $T$ and
any point $(y_1,\ldots,y_{i-1},y_{i+1}, \ldots,y_{j-1},y_{j+1}, \ldots, y_n) \in T^{n-2}$.
Let $\V$ be any component of the subset
\begin{eqnarray*} \{(\nu^{\langle n\rangle}_{g,m})^{-1} (y_1,\ldots,y_{i-1},p,y_{i+1}, \ldots,y_{j-1},q,y_{j+1}, \ldots, y_n) \; | p+q \in \mathfrak{g}\}.
\end{eqnarray*}
Then the restriction map
\[\V \longrightarrow \Sym^{L \cdot T}(T)_m \]
is finite and surjective.
\end{lemma}

\begin{proof}
Assume $L \cdot \e_i>0$ for some $i$. Varying $p$, we obtain a one-dimensional nontrivial family of surfaces $\Bl_{y_1,\ldots,y_{i-1},p,y_{i+1}, \ldots,y_n}(P)$
and a one-dimensional non--constant family of line bundles whose
restrictions to $T$ yield a one-dimensional non--constant family of
line bundles. This together with Lemma \ref{lemma:P3}(ii) yields  (i).

Finally, assume $L \cdot \e_i=a_i>0$ and $L \cdot \e_j=a_j>0$, with $a_i \neq a_j$.   Varying $p+q \in \mathfrak{g}$, we get a one-dimensional nontrivial family of surfaces $\Bl_{y_1,\ldots,y_{i-1},p,y_{i+1}, \ldots,y_{j-1},q,y_{j+1}, \ldots,y_n}(P)$ as above and a one-dimensional family of line bundles, all of the form
$L'-a_i\e_i-a_j\e_j$, with $L'$ fixed (on $\Bl_{y_1,\ldots,y_{i-1},y_{i+1}, \ldots,y_{j-1},y_{j+1}, \ldots,y_n}(P)$) and $\e_i$ and $\e_j$ varying with $p$ and $q$. To prove  (ii), we will prove that the family of restrictions
\[ \left\{L'|_T-a_ip-a_jq\right\}_{p+q \in \mathfrak{g}} \]
to $T$ is non--constant. Assume that
\begin{equation} \label{eq:movl'}
  L'|_T-a_ip-a_jq \sim L'|_T-a_ip'-a_jq' \; \; \mbox{for} \; \; p+q \neq p'+q' \in \mathfrak{g}.
  \end{equation}
This yields $(a_i-a_j) p \sim (a_i-a_j) p'$ and  $(a_i-a_j) q \sim (a_i-a_j) q'$. For fixed $x \in T$, there are only finitely many points $y \in T$ such that $(a_i-a_j) x \sim (a_i-a_j) y$. For general $p+q, p'+q' \in  \mathfrak{g}$, condition \eqref{eq:movl'} is therefore not fulfilled. This finishes the proof of (ii).
\end{proof}

The main existence result of this subsection is the following:

\begin{prop} \label{prop:P2}
  Let $y_1, \ldots, y_n \in T$ be general, $n \leq 8$, and $L$ be big and nef on
  $\widetilde{P}$.
  If  $ 0< m \leq  L \cdot T-3$,
  the variety
  $V^*_{0,m}(\widetilde{P},T,L)$ is nonempty of dimension $T\cdot L-m$. Moreover, its general member intersects any fixed curve on $\widetilde{P}$ different from $T$ transversely.
\end{prop}

\begin{proof}
This is an application of \cite[Cor. 3.11]{CGL};
there are some conditions to check, so we give a proof for
completeness.

 The statements  about dimension and transversal intersection follow from
 Proposition \ref{prop:TD}(ii) and (iii), respectively, once non--emptiness is proved. By Proposition \ref{prop:TD}(iii)--(iv)
 we have that $V^*_{0,m}(\widetilde{P},L) \neq \emptyset$ as soon as  $V_{0,m}(\widetilde{P},T,L) \neq \emptyset$, because of the condition $m \leq  L \cdot T-3$. We therefore have left to prove nonemptiness of
  $V_{0,m}(\widetilde{P},T,L)$.  We will prove this by induction on $m$, as in the proof of Proposition \ref{prop:R2}, again following an idea in the proof of \cite [Thm. 3.10]{CGL}.

  Since $y_1\ldots,y_n \in T$ are general and $n \leq 8$, we may take $y_1,\ldots,y_n$ to be general points of $\PP^2$ and $T$ a general plane cubic containing them. Hence  $\widetilde{P}$ is a Del Pezzo surface, so that $T$ is ample on it. 
It is then well-known,  by \cite[Thms. 3-4]{GLS}, that $V_{0,1}(\widetilde{P},T,L) \neq \emptyset$.

Assume now that we have proved non--emptiness of $V_{0,m}(\widetilde{P},T,L)$ for some $1 \leq m \leq  L \cdot T-4$.  By Lemma \ref{lemma:P3}(ii) its general member $C$ satisfies
  \[ C \cap T=mp_0+p_1+\cdots+p_l+p_{l+1}+p_{l+2}, \; \; l=L \cdot T-m-2 \geq 2,\]
  where $p_0,\ldots,p_{l+2}$ are distinct, and we may take $p_0,\ldots,p_{l+1}$ general on $T$. 
  
  For later purposes, we observe that, 
since there are only finitely many divisor classes $D$ such that $C-D
>0$
(by which we mean that $C-D$ is effective and non-zero),
hence $C\cdot T>D\cdot T$, and for each such class the image of the restriction morphism $|D| \to \Sym^{D \cdot T} (T)$ has codimension one,  
the generality of the points implies that
  \begin{eqnarray}
    \label{eq:noD}
    \mbox{there is no effective divisor $D\neq T$ such that} \\
    \nonumber  C-D>0 \; \; \mbox{and} \; \;
   D \cap T \subset \{p_0,p_1,\ldots,p_l\}.
  \end{eqnarray}

  Set $\d=mp_0+p_1+\cdots+p_{l}$. Then $V_{0,\d}(\widetilde{P},T, L) \neq \emptyset$ and all its components are one-dimensional, by Proposition \ref{prop:TD}(i). The general member in any component intersects $T$ in $mp_0+p_1+\cdots+p_{l}+q_1+q_2$, where the points $q_1,q_2$ vary in the family,  by Proposition \ref{prop:TD}(iii). 
Pick a component $\overline{V}$ of its closure inside the component of the Hilbert scheme of $\widetilde{P}$ containing $|L|$. After a finite base change, we find a smooth projective curve $B$, a surjection $B \to \overline{V}$ and a family
  \[ \xymatrix{ \C \ar[r]^{f}  \ar[d]_{g} & \widetilde{P} \\
      B} \]
  of stable maps of genus zero such that, setting $\C_b:=g^*b$ for any $b \in B$, the curve
$f_*\C_b$ is a member of $\overline{V}$, and such that 
  \[ f^*T=mP_0+P_1+\cdots+P_l+Q_1+Q_2+W,\]
  where
  \begin{itemize}
  \item[(I)] $P_i$ and $Q_j$ are sections of $g$, for $i \in \{0,\ldots,l\}$, $j \in \{1,2\}$,
  \item[(II)] $f(P_i)=p_i$, for $i \in \{0,\ldots,l\}$,
\item[(III)] $f(Q_j)=T$, for $j \in \{1,2\}$,
\item[(IV)] $g_*W=0$,
  \item[(V)] $f_*W=0$;
  \end{itemize}
  the latter property follows from the fact that no member of
  $\overline{V}$ contains $T$, by Lemma \ref{lemma:P3}(i).
% Property (III) implies as in the proof of Proposition~\ref{prop:R2}
% that $f^{-1}(T)$ is either connected, or has two connected components
% containing $Q_1$ and $Q_2$ respectively.
Since $T$ is ample, $f^*T$ is big and nef, hence its support
$f^{-1}(T)$ is connected as a consequence of Kawamata--Viehweg
vanishing. 
Therefore $P_0$ and $Q_2$ are connected by a chain $W' \subset W$ such that
$W' \subset \C_{b_0}$ for some $b_0 \in B$, and $f(W')=p_0$. Thus, 
  \begin{equation}\label{eq:pipi2} f_* \C_{b_0} \cap T=(m+1)p_0+p_1+\cdots+p_l+q_1, \; \; q_1=f(Q_1), \; \; l \geq 2.\end{equation} 
  Since $T$ is ample, all components of $f_* \C_{b_0}$ intersect
  $T$. By \eqref{eq:noD} and \eqref{eq:pipi2}, $f_* \C_{b_0}$ must be
  reduced and irreducible, say $ f_* \C_{b_0}=C$, an irreducible
  rational curve. Since $\C_{b_0}$ has arithmetic genus $0$, it consists of a tree of  smooth rational curves, with one component $\widetilde{C}$ such that $f(\widetilde{C})=C$, and the other components contracted by $f$. Therefore, $f^{-1}(p_0) \cap \widetilde{C}=(W'+P_0) \cap \widetilde{C}$ is a single (smooth) point of $\widetilde{C}$. It follows that
  $[C] \in V_{0,(m+1)p_0+p_1+\cdots+p_{l}}(\widetilde{P},T, L)$, which implies  $[C] \in V_{0,m+1}(\widetilde{P},T,L)$.  \end{proof}

\section{Deforming to rigid elliptic curves} \label{sec:prepar}

As mentioned in the introduction, to prove Theorem \ref{thm:main} it will suffice  by \cite[Cor. 1]{indam} to prove that 
$V_{|L|,g-1}(S)$ has a $0$-dimensional component. We will call any element of 
such a $0$-dimensional component a {\it rigid nodal elliptic curve}.
We will prove the existence of such a curve on a general $(S,L)$ in any component of $\E_g \setminus \E_g[2]$ by degeneration, 
using Theorem \ref{thm:deform}, constructing suitable curves on limit
surfaces in $\mathcal D^*$ that will deform to rigid curves in $V_{|L|,g-1}(S)$.
In this section we will identify numerical conditions on limit line bundles under which deformations to such curves can be achieved.
The general strategy of proof is given in the following:

\begin{prop} \label{prop:GS}
  Let $X=\widetilde{R} \cup_T \widetilde{P}$ be a general member of a component of $\D^*$ and $Y=C \cup_T D$ a curve on $X$, with $C \subset \widetilde{R}$ and $D \subset \widetilde{P}$, having the following properties: there are distinct points $x,p_1,\ldots,p_k,q_1,\ldots,q_l$ on $T$, for nonnegative integers $k,l$, and a positive integer $m$ such that
  \begin{itemize}
  \item $C=C_0+C_1+ \cdots+ C_l$ is nodal, with  $[C_0] \in V^{**}_{1,m}(\widetilde{R},T, C_0)$  and 
  $C_i$ a $(-1)$-curve, $i \in \{1,\ldots,l\}$,
  \item $C_0 \cap T=mx+p_1+\cdots+ p_k$,
  \item $C_0$ is odd, 
  \item $C_i \cap T=\{q_i\}$, $i \in \{1,\ldots,l\}$, 
    \end{itemize}
    and
\begin{itemize}
\item $D=D_0+D_1+ \cdots+ D_k$ is nodal, with
$[D_0] \in V^{*}_{0,m}(\widetilde{P},T, D_0)$  and
  $D_i$  a $(-1)$-curve, $i \in \{1,\ldots,k\}$,
  \item $D_0 \cap T=mx+q_1+\cdots+ q_l$,
   \item $D_i \cap T=\{p_i\}$, $i \in \{1,\ldots,k\}$.   
    \end{itemize}
Then $Y$ deforms to an irreducible rigid nodal elliptic curve on the general deformation $S$ of $X$.
\end{prop}

\begin{proof}
Here is a picture of what $Y$ looks like: 
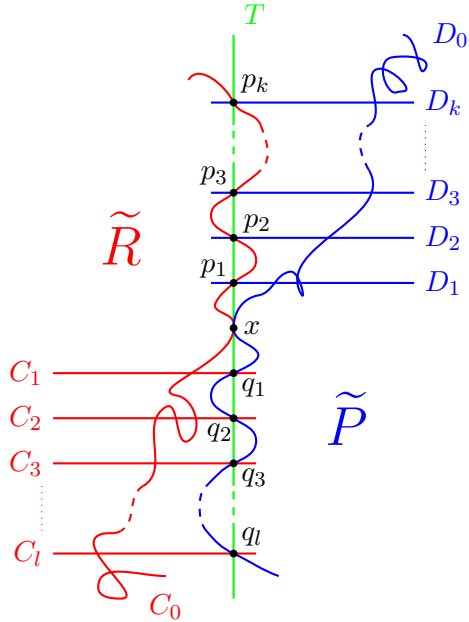
\begin{center}
\begin{tikzpicture}[scale=0.3]

\coordinate (x) at (0,0);
\coordinate (p1) at (0,2);
\coordinate (p2) at (0,4);
\coordinate (p3) at (0,6);
\coordinate (pk) at (0,10);
\coordinate (q1) at (0,-2);
\coordinate (q2) at (0,-4);
\coordinate (q3) at (0,-6);
\coordinate (ql) at (0,-10);

\draw[green,dashed,thick] (0,7) -- (0,9);
\draw[green,dashed,thick] (0,-9) -- (0,-7);
\draw[green,thick] (0,-7) -- (0,7);
\draw[green,thick] (0,9) -- (0,13) node[above right] {$T$};
\draw[green,thick] (0,-12) -- (0,-9);

\draw[blue,thick] (-1,2) -- (8,2) node[right] {$D_1$};
\draw[blue,thick] (-1,4) -- (8,4) node[right] {$D_2$};
\draw[blue,thick] (-1,6) -- (8,6) node[right] {$D_3$};
\draw[blue,thick] (-1,10) -- (8,10) node[right] {$D_k$};

\draw[blue,dotted] (8.5,7) -- (8.5,9);

\draw[red,thick] (1,-2) -- (-8,-2) node[left] {$C_1$};
\draw[red,thick] (1,-4) -- (-8,-4) node[left] {$C_2$};
\draw[red,thick] (1,-6) -- (-8,-6) node[left] {$C_3$};
\draw[red,thick] (1,-10) -- (-8,-10) node[left] {$C_l$};
\draw[red,dotted] (-8.5,-7) -- (-8.5,-9);

\draw[smooth,red,thick]
(-2,11) [out=60,in=115]  to (pk) to[out=-65,in=110] (1.2,9)
(1.2,7) to[out=-135,in=45] (p3) to[out=-135,in=90] (-1,5) to[out=-90,in=145] (p2) to[out=-35,in=90] (1,3) to[out=-90,in=45] (p1) to[out=-135,in=65]
(-0.8,1.1)  to[out=-115,in=90]
(0,0.2) to[out=-90,in=90]
(x) to[out=-90,in=170] (-2.5,-3)
to[out=-10,in=30] (-2,-5) to[out=-150,in=30] (-3,-3.5) to[out=-150,in=30] (-4,-7) (-5,-9) to[out=-150,in=30] (-6,-9.2) to[out=210,in=0] (-5,-12) to[out=180,in=-90] (-6,-11.5) to[out=90,in=180] (-3,-11) node[below=0.1cm] {$C_0$}
;

\draw[smooth,red,dashed, thick]
(-4,-7) to[out=-150,in=30] (-5,-9);

\draw[smooth,red,dashed, thick]
(1.2,9) to[out=-70,in=45] (1.2,7);

\draw[smooth,blue,thick]
(2,-11)[out=160,in=-25]  to (ql) to[out=155,in=-50] (-1.2,-9)
(-1.2,-7) to[out=45,in=-155] (q3) to[out=25,in=-90] (1,-5) to[out=90,in=-35] (q2) to[out=145,in=-90] (-1,-3) to[out=90,in=-155] (q1) to[out=25,in=-45]
(1,-1)  to[out=135,in=-90] (x)  to[out=90,in=-170]
(1.1,1.45) to[out=10,in=-170]
(2.5,2.5) to[out=10,in=-60] (3,1.5) to[out=120,in=-60] (6,7) (6,9) to[out=30,in=180] (6.5,11.5)
to[out=0,in=0] (7.5,10.5) to[out=180,in=-40] (7,10.5) to[out=140,in=-170] (7,12) to[out=10,in=60] (8.5,11.5) to[out=-120,in=0] (7.5,13) node[right=0.25cm] {$D_0$};

\draw[smooth,blue,dashed,thick]
(-1.2,-9)
to[out=130,in=-90] (-1.5,-8) to[out=90,in=-135] (-1.2,-7);

\draw[smooth,blue,dashed,thick]
(6,7) to[out=120,in=-150] (6,9); 

\filldraw (x) circle (0.15);
\filldraw  (q1) circle (0.15);
\filldraw (q2) circle (0.15);
\filldraw (q3) circle (0.15);
\filldraw (ql) circle (0.15);
\filldraw (p1) circle (0.15);
\filldraw (p2) circle (0.15);
\filldraw (p3) circle (0.15);
\filldraw (pk) circle (0.15);

\draw (0,0) node[right] {$x$};
\draw (-0.1,-1.8) node[below right] {$q_1$};
\draw (0.4,-3.8) node[below left] {$q_2$};
\draw (-0.1,-5.8) node[below right] {$q_3$};
\draw (-0.1,-10.1) node[above right] {$q_l$};

\draw (0.2,1.8) node[above left] {$p_1$};
\draw (0,3.8) node[above right] {$p_2$};
\draw (0.2,5.9) node[above left] {$p_3$};
\draw (-0.1,10) node[above right] {$p_k$};

\draw[red] (-5,4) node {\huge$\widetilde{R}$};
\draw[blue] (5,-4) node {\huge$\widetilde{P}$};
\end{tikzpicture}

\end{center}
We note that $Y$ is Cartier,   has an $m$--tacnode at $x$ and is otherwise nodal. Moreover, an easy computation as in \cite[p.~119]{kn-JMPA} shows that
\begin{equation}
  \label{eq:inters}
  \dim (|Y|)=\frac{1}{2}Y^2=p_a(Y)-1.
\end{equation}

 We define 
  \begin{itemize}
  \item $N_{C_0}$ the scheme of nodes of $C_0$ and $\gamma_0$ its degree, 
  \item $N_{C}$ the scheme of 
intersection points between components of $C$ and
    $\gamma$ its degree,
    \item $N_{D_0}$ the scheme of nodes of $D_0$ and $\delta_0$ its degree,
    \item $N_{D}$ the scheme of 
intersection points between components of $D$ and
$\delta$ its degree.
      \end{itemize}
Then $N_{C_0} \cup N_C \cup N_{D_0}\cup N_D$ is the set of nodes of $Y$ off $T$.
Since 
\[ \gamma_0=p_a(C_0)-1=\frac{1}{2}C_0 \cdot (C_0+K_{\widetilde{R}})= \frac{1}{2}\left(C_0^2-C_0 \cdot T\right)= \frac{1}{2}\left(C_0^2-m-k\right)
      \]
and, similarly,
      \[
        \delta_0=p_a(D_0)=\frac{1}{2}\left(D_0^2-m-l\right)+1,
      \]
we compute
      \begin{eqnarray}
   \label{eq:gan}     p_a(Y)& = & \frac{1}{2}Y^2+1=\frac{1}{2}\left(C^2+D^2\right)+1=\frac{1}{2}\left(C_0^2-l+2\gamma+D_0^2-k+2\delta\right)+1 \\
\nonumber               & = & \frac{1}{2}\left(C_0^2-m-k\right) +\frac{1}{2}\left(D_0^2-m-l\right) +m+\gamma+\delta+1 \\
 \nonumber        & = & \gamma_0+(\delta_0-1)+m+\gamma+\delta+1 = \gamma_0+\gamma+\delta_0+\delta+m.
      \end{eqnarray}

      Let $\mathfrak{X} \to \DD$  be the deformation of $X$  to a general smooth Enriques surface $S_t$ in Theorem \ref{thm:deform}. Then $Y$ deforms to a Cartier divisor $Y_t$ on $S_t$ by the same theorem. By \eqref{eq:inters} we have $\dim (|Y|)=\dim(|Y_t|)$. Let $\mathfrak{D}$ be the sublinear system of $|Y|$ of curves with an $(m-1)$--tacnode at $x$ and passing through $N_{C_0} \cup N_C \cup N_{D_0}\cup N_D$. 
We claim that
\begin{equation}
  \label{eq:pergk}
  \mbox{$\mathfrak{D}$ consists only of $Y$ itself.}\footnote{ From a deformation-theoretic point of view, \eqref{eq:pergk} implies that the {\it equisingular deformation locus of $Y$ in $\mathfrak{X}$} is smooth and zero-dimensional (cf.\ \cite[Lemma 3.4]{gk}), thus
consisting only of the point $[Y]$.}
\end{equation}
Granting this for the moment, \eqref{eq:inters}-\eqref{eq:pergk} yields that the codimension of $\mathfrak{D}$ is $m-1+\gamma_0+\gamma+\delta_0+\delta$. Thus, the hypotheses of \cite[Thm. 3.3, Cor. 3.12 and proof of Thm. 1.1]{gk} are fulfilled\footnote{We remark that the hypothesis in \cite{gk} that both components of $X$ are regular is not necessary; it suffices that $h^1(\O_X)=0$, which is proved as in \cite[Lemma 3.4]{kn-JMPA}.}
and we conclude that, under the deformation of $X$ to $S_t$, we may deform $Y$   in such a way that  the $m$-tacnode of $Y$ at $x$  deforms  to
      $m-1$ nodes  and  the  $\gamma_0+\gamma+\delta_0+\delta$ nodes of $Y$ in the smooth locus of $X$  are preserved,  whereas the nodes of $Y$ on $T$ automatically smooth. Thus $Y$ deforms to a  nodal  curve $Y_t \subset S_t$ with
  a total of 
 \[ \gamma_0+\gamma+\delta_0+\delta+m-1=p_a(Y)-1=p_a(Y_t)-1 \]
 nodes (using \eqref{eq:gan}). 
 Since one easily sees that no subcurve of $Y$ is Cartier, $Y_t$ is irreducible, whence $Y_t$ is nodal and elliptic, as desired. It is rigid, as $Y$ is rigid on $X$.

 We have left to prove \eqref{eq:pergk}. To this end, let $A \cup_T B \in |Y|$ be a curve with an $(m-1)$--tacnode at $x$ and passing through $N_{C_0} \cup N_C \cup N_{D_0}\cup N_D$, where $A \subset \widetilde{R}$ and  $B \subset \widetilde{P}$. Then both $A$ and $B$ must intersect $T$ in a scheme containing $(m-1)x$ and moreover $N_{C_0} \cup N_C \subset A$ and $N_{D_0} \cup N_D \subset B$.

  The fact that $N_D \subset B$ implies that $B$ must contain all $(-1)$--curves $D_1,\ldots,D_k$. Hence $B =B_0+D_1+\cdots+ D_k$, with $B_0 \sim D_0$. Similarly,
  the fact that $N_C \subset A$ implies that $A$ must contain all $(-1)$--curves $C_1,\ldots,C_k$. Hence $A =A_0+C_1+\cdots +C_k$, with $A_0 \sim C_0$. Since
  $A \cup_T B$ is Cartier, $A_0$ must pass through $p_1,\ldots,p_k$ and
  $B_0$ must pass through $q_1,\ldots,q_l$. Thus
  \[ A_0 \cap T \supset (m-1)x+p_1+\cdots+p_k=:Z_C \; \; \mbox{and} \; \;
    B_0 \cap T \supset (m-1)x+q_1+\cdots+q_l=:Z_D.\]
  Hence $A_0 \in |\O_{\widetilde{R}}(C_0) \* \I_{N_{C_0} \cup Z_C}|$. Since
$\deg(Z_C)=C_0 \cdot T-1$ and
$[C_0] \in V^{**}_{1,m}(\widetilde{R},T, C_0)$, this implies $A_0=C_0$ (recall Definition \ref{def:dustelle}), whence $A=C$.

    Similarly, $B_0 \in |\O_{\widetilde{P}}(D_0) \* \I_{N_{D_0} \cup Z_D}|$, whence, if $B_0 \neq D_0$, we would get 
      \begin{eqnarray*} D_0^2 & = & B_0 \cdot D_0 \geq \deg (Z_D) +2\deg(N_{D_0}) = (D_0 \cdot T-1)+
                                    2p_a(D_0)\\
        & = &
        -D_0 \cdot K_{\widetilde{P}}-1+D_0^2+D_0 \cdot K_{\widetilde{P}}+2 >D_0^2,
        \end{eqnarray*}
        a contradiction. Thus $B_0=D_0$, whence $B=D$. This proves
        \eqref{eq:pergk}. 
\end{proof}

The next two results form the basis for our proof of Theorem \ref{thm:main}. We henceforth assume that $E$ is a {\it general} elliptic curve.

  \begin{prop} \label{prop:metodo1}
    Let $R'$ (respectively, $P'$) be a blow--up of $R$ (resp., $P$) at $s \geq 0$ (resp., $t \geq 1$) general points of $T$. Assume $L'$ (resp., $L''$) is a line bundle (or Cartier divisor) on  $R'$ (resp., $P'$) and $k$ is an integer such that the following conditions are satisfied:
    \begin{itemize}
    \item[(i)]   $s+t-5\leq k \leq \min\{3,s,t-1\}$, 
     \item[(ii)] $L'\cdot T=L''\cdot T$,
    \item[(iii)]  $L' \equiv L'_0+C_1+ \cdots +C_k$, where the $C_i$ are disjoint $(-1)$-curves and $L'_0$ satisfies condition $(\star)$ and is odd,
    \item[(iv)] $L'' \sim L''_0+D_1+ \cdots +D_k$, where the $D_i$ are disjoint $(-1)$-curves and $L''_0$ is big and nef,
    \item[(v)] there are $t-k$ additional $(-1)$-curves on $P'$,
mutually disjoint and disjoint from $D_1,\ldots,D_k$,
 such that $L''$ is positive on at least one of them.
    \end{itemize}

  Then there are blow--ups $\widetilde R\to R'$ and $\widetilde P\to P'$ at distinct point of $T$ such that $\widetilde R\cup_T\widetilde P$ is general in a component of $\mathcal D^*$ and, denoting by $\widetilde L'$ and $\widetilde L''$  the pull--backs of $L'$ and $L''$ to $\widetilde R$ and $\widetilde P$ respectively, there is a line bundle $\widetilde L\in [\widetilde L',\widetilde L'']$ (cf.\ Remark \ref {rem:ingcro10}) such that 
  $(\widetilde R \cup_T \widetilde P, \widetilde L)$ deforms to a smooth  polarized Enriques surface $(S,L)$ with $S$ containing an irreducible, rigid nodal elliptic curve in $|L|$.
\end{prop}

\begin{proof}
Set $m:=L_0'\cdot T-3=L''_0\cdot T-3$ (equality follows from assumptions (ii)--(iv)). By condition (iv) in $(\star)$, we have $m>0$.
The line bundle $L'_0$ satisfies the conditions of Proposition
\ref{prop:R2}. Hence $V^{**}_{1,m}(R',T, L'_0) \neq \emptyset$ and its
general member intersects $C_1,\ldots,C_k$ transversely.
Let $C_0$ be such a general member;
we have $L' \equiv C_0+C_1+ \cdots +C_k$ by assumption (iii).

  The surface $R'$ is a blow--up of $R$ at $s$ general points $y_1,\ldots, y_s$ of $T$, where $k\leq s\leq 4$ by assumption (i). Denoting the exceptional divisor over $y_i$ by $\e_i$, the surface $R'$ contains precisely $s$ additional $(-1)$-curves $\e'_i$,
$i=1,\ldots,s$, such that $\e_i \cdot \e'_i=1$ and each $\e_i + \e'_i$ is a fiber    of the projection $R'\to E$. Set $y'_i:=\e'_i \cap T$. 
Since $C_1, \ldots, C_k$ are disjoint $(-1)$-curves by assumption (iii), we have, after renumbering, that $C_i=\e_i$ or $\e'_i$. We will in the following for simplicity assume that $C_i=\e_i$; the other cases  can be treated in the same way by substituting $y_i$ with $y'_i$ at the appropriate places. Thus we have
\begin{eqnarray*}
  C_0 \cap T & = & mp+p_1+p_2+p_3, \; \; \mbox{for} \; \; p,p_1,p_2,p_3 \in T, \\
C_i \cap T & = & y_i, \; \; i \in \{1,\ldots,k\}.  
\end{eqnarray*}
By Lemma \ref{lemma:R3}(ii) the points $p,p_1,p_2,p_3$ are general on $T$ (even for fixed $y_1,\ldots,y_k$). The points $y_1,\ldots,y_k$ are general on $T$ by the assumption that $R'$ is a blow--up of $R$ at {\it general} points on $T$.

The surface $P'$ is a blow--up of $\PP^2$ at $t$ general points on
$T$, where $k<t\leq 5$ by assumption (i). The line bundle $L''_0$ is
big and nef by assumption (iv), whence by Proposition \ref{prop:P2} we
have $V^*_{0,m}(P',T,L''_0) \neq \emptyset$ and its general member
intersects $D_1,\ldots,D_k$ transversely. Let $D_0$ be such a general
member; we have $L'' \sim D_0+D_1+ \cdots +D_k$ by assumption (iv).

 By assumption (iv)-(v), $D_1,\ldots,D_k$ are disjoint $(-1)$-curves belonging to a set of $t$ disjoint $(-1)$-curves, which we may therefore take as an exceptional set for a blowdown $P' \to \PP^2$ centered at points that we denote by $x_1,\ldots, x_t \in T$. Furthermore, there is an exceptional curve, different from $D_1,\ldots,D_k$, call it  $D_{k+1}$, such that $D_{k+1} \cdot L''_0>0$. We have
\begin{eqnarray*}
  D_0 \cap T & = & mq+q_1+q_2+q_3, \; \; \mbox{for} \; \; q,q_1,q_2,q_3 \in T, \\
D_i \cap T & = & x_i, \; \; i \in \{1,\ldots,k\}.  
\end{eqnarray*}
 Since $D_{k+1} \cdot L''_0>0$, we can apply Lemma \ref{lemma:P3-family}(i), which ensures that, even for fixed $x_1,\ldots, x_k$, by moving $x_{k+1}$, the intersection $D_0 \cap T$ is a general element in $\Sym^{L''_0\cdot T}(T)_m$. Hence the points $q,q_1,q_2,q_3$ are general on $T$.  The points $x_1,\ldots,x_k$ are general on $T$ by the assumption that $P'$ is a blow--up of $P$ at {\it general} points on $T$.

Since the points $p,p_1,p_2,p_3,y_1,\ldots,y_k$ are general on $T$,
and likewise  the points $q,q_1,q_2,q_3,x_1,\ldots,x_k$, we may assume
that our choices of $C_0$ and $D_0$ come  with  the identifications
\begin{eqnarray*}
  & p  =  q \\
  & x_i=p_i, \; \; y_i=q_i \; \; \mbox{for} \; \; 1 \leq i \leq k \quad \text {(recall that $k\leq 3$),}& \\
  & p_i=q_i \; \; \mbox{for} \; \; k+1 \leq i \leq 3 \quad \text {(if $k= 3$, this condition is empty)}, 
\end{eqnarray*}
so that we can glue $R'$ and $P'$ along $T$ in such a way that
 $(C_0+C_1+ \cdots +C_k,D_0+D_1+ \cdots +D_k)$ is Cartier on $R' \cup_T P'$. The following picture shows the case $k=2$: 

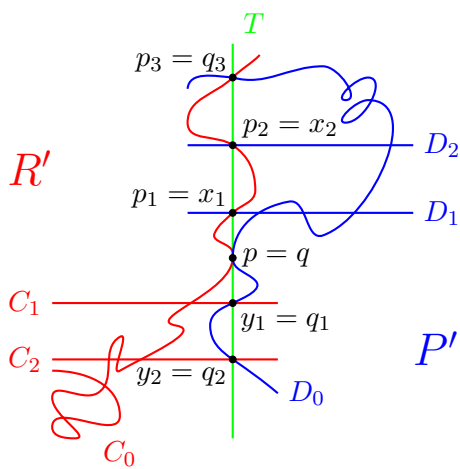
\begin{center}
\begin{tikzpicture}[scale=0.3]

\coordinate (x) at (0,0);
\coordinate (p1) at (0,2);
\coordinate (p2) at (0,5);
\coordinate (p3) at (0,8);
\coordinate (q1) at (0,-2);
\coordinate (q2) at (0,-4.5);

\draw[green,thick] (0,-8) -- (0,9.5) node[above right] {$T$};

\draw[blue,thick] (-2,2) -- (8,2) node[right] {$D_1$};
\draw[blue,thick] (-2,5) -- (8,5) node[right] {$D_2$};

\draw[red,thick] (2,-2) -- (-8,-2) node[left] {$C_1$};
\draw[red,thick] (2,-4.5) -- (-8,-4.5) node[left] {$C_2$};

\draw[smooth,red,thick]
(1.2,9) to[out=-135,in=45] (p3) to[out=-135,in=90] (-2,6) to[out=-90,in=145] (p2) to[out=-35,in=90] (1,3) to[out=-90,in=45] (p1) to[out=-135,in=65]
(-0.8,1.1)  to[out=-115,in=90]
(0,0.2) to[out=-90,in=90]
(x) to[out=-90,in=170] (-2.5,-3)
to[out=-10,in=30] (-5,-5) to[out=-150,in=30] (-5,-3.5)
to[out=-150,in=30]
(-6,-6)
to[out=-150,in=90] (-8,-6.2) to[out=-90,in=0] (-7,-8) to[out=180,in=-90] (-8,-7.5) 
 to[out=90,in=-90] (-5,-7) node[below=0.2cm] {$C_0$}
 to[out=90,in=0] (-8,-5)
;

\draw[smooth,blue,thick]
(2,-6)  node[right] {$D_0$} to[out=140,in=-35] (q2) to[out=145,in=-90] (-1,-3) to[out=90,in=-155] (q1) to[out=25,in=-45]
(1,-1)  to[out=135,in=-90] (x)  to[out=90,in=-170]
(2.5,2.5) to[out=10,in=-160] (4,1) to[out=20,in=0] (6,7)
to[out=180,in=180] (5,6)
to[out=0,in=0] (6,8)
to[out=180,in=180] (5,7)
to[out=0,in=0] (4,8.5)
to[out=180,in=-15] (p3)  
to[out=175,in=85] (-2,7.5);

\filldraw (x) circle (0.15);
\filldraw (q1) circle (0.15);
\filldraw (q2) circle (0.15);
\filldraw (p1) circle (0.15);
\filldraw (p2) circle (0.15);
\filldraw (p3) circle (0.15);

\draw (0,0) node  [right] {$p=q$};
\draw (-0.1,-2) node [below right] {$y_1=q_1$};
\draw (0.2,-4.4) node [below left] {$y_2=q_2$};
\draw (0.2,1.8) node [above left] {$p_1=x_1$};
\draw (0,4.9) node  [above right] {$p_2=x_2$};
\draw (0.2,7.8) node  [above left] {$p_3=q_3$};

\draw[red] (-9,4) node {\LARGE$R'$};
\draw[blue] (9,-4) node {\LARGE$P'$};

\end{tikzpicture}

\end{center}

By assumption (i) we have that $s+t \leq k+5 \leq 8$. The set of points
\[ y_1,\ldots,y_s,x_1,\ldots,x_t, (p_{k+1}=q_{k+1}), \ldots, (p_{3}=q_{3})\]
is therefore a set of $s+t+3-k \leq 8$ general points on $T$. Pick now points
$w_1,\ldots, w_{k+6-s-t}$ on $T$, general under the condition that
\[ y_1+\cdots+y_s+x_1+\cdots+x_t+  p_{k+1} + \cdots+ p_{3}+ w_1 + \cdots+ w_{k+6-s-t} \in |\N_{T/R} \* \N_{T/P}|.
\]
Note that in this way we get a general divisor of $|\N_{T/R} \* \N_{T/P}|$.
Then blow up either $R'$ or $P'$ at the points 
\[ (p_{k+1}=q_{k+1}), \ldots, (p_{3}=q_{3}), w_1,\ldots, w_{k+6-s-t}\]
to obtain $\widetilde{R} \cup_T \widetilde{P}$, which is a general member of a component of $\D^*$, containing the inverse image $Y$ of the curve $(C_0+C_1+ \cdots +C_k) \cup_T (D_0+D_1+ \cdots +D_k)$, which satisfies the conditions of Proposition \ref{prop:GS}. The result therefore follows by Proposition \ref{prop:GS}.
(The following picture shows the inverse image of the curve in the previous image in the case $k=2$: an additional $(-1)$-curve appears over $p_3=q_3$.)

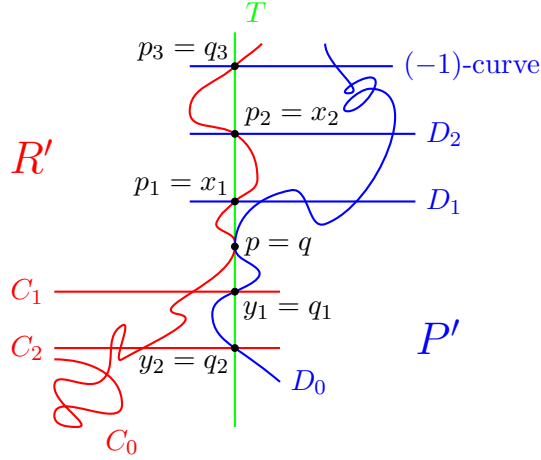
\begin{center}
\begin{tikzpicture}[scale=0.3]

\coordinate (x) at (0,0);
\coordinate (p1) at (0,2);
\coordinate (p2) at (0,5);
\coordinate (p3) at (0,8);
\coordinate (q1) at (0,-2);
\coordinate (q2) at (0,-4.5);

\draw[green,thick] (0,-8) -- (0,9.5) node[above right] {$T$};

\draw[blue,thick] (-2,2) -- (8,2) node[right] {$D_1$};
\draw[blue,thick] (-2,5) -- (8,5) node[right] {$D_2$};

\draw[red,thick] (2,-2) -- (-8,-2) node[left] {$C_1$};
\draw[red,thick] (2,-4.5) -- (-8,-4.5) node[left] {$C_2$};

\draw[smooth,red,thick]
(1.2,9) to[out=-135,in=45] (p3) to[out=-135,in=90] (-2,6) to[out=-90,in=145] (p2) to[out=-35,in=90] (1,3) to[out=-90,in=45] (p1) to[out=-135,in=65]
(-0.8,1.1)  to[out=-115,in=90]
(0,0.2) to[out=-90,in=90]
(x) to[out=-90,in=170] (-2.5,-3)
to[out=-10,in=30] (-5,-5) to[out=-150,in=30] (-5,-3.5)
to[out=-150,in=30]
(-6,-6)
to[out=-150,in=90] (-8,-6.2) to[out=-90,in=0] (-7,-8) to[out=180,in=-90] (-8,-7.5) 
 to[out=90,in=-90] (-5,-7) node[below=0.2cm] {$C_0$}
 to[out=90,in=0] (-8,-5)
;

\draw[smooth,blue,thick]
(-2,8) -- (7,8) node[right] {$(-1)$-curve};

\draw[smooth,blue,thick]
(2,-6) node[right] {$D_0$} to[out=140,in=-35] (q2) to[out=145,in=-90] (-1,-3) to[out=90,in=-155] (q1) to[out=25,in=-45]
(1,-1)  to[out=135,in=-90] (x)  to[out=90,in=-170]
(2.5,2.5) to[out=10,in=-160] (4,1) to[out=20,in=0] (6,7)
to[out=180,in=180] (5,6)
to[out=0,in=0] (6,7.5)
to[out=180,in=180] (5,6.8)
to[out=0,in=-90] (4,9);

\filldraw (x) circle (0.15);
\filldraw (q1) circle (0.15);
\filldraw (q2) circle (0.15);
\filldraw (p1) circle (0.15);
\filldraw (p2) circle (0.15);
\filldraw (p3) circle (0.15);

\draw (0,0) node  [right] {$p=q$};
\draw (-0.1,-2) node [below right] {$y_1=q_1$};
\draw (0.2,-4.4) node [below left] {$y_2=q_2$};
\draw (0.2,1.8) node [above left] {$p_1=x_1$};
\draw (0,4.9) node  [above right] {$p_2=x_2$};
\draw (0.2,7.8) node  [above left] {$p_3=q_3$};

\draw[red] (-9,4) node {\LARGE$R'$};
\draw[blue] (9,-4) node {\LARGE$P'$};

\end{tikzpicture}
\end{center}
\end{proof}

\begin{prop} \label{prop:metodo2}
  Let $\widetilde{R}$ (respectively, $\widetilde{P}$) be a blow--up of $R$ (resp., $P$) at $4$ (resp., $5$) general points of $T$. Assume $L'$ (resp., $L''$) is a line bundle (or Cartier divisor) on  $\widetilde{R}$ (resp.,
    $\widetilde{P}$) such that the following conditions are satisfied:
    \begin{itemize} 
\item[(i)] $L' \cdot T=L'' \cdot T $,
    \item[(ii)]  $L' \equiv L'_0+C_1+ C_2 +C_3$, where the $C_i$ are disjoint $(-1)$-curves and $L'_0$ satisfies condition $(\star)$ and is odd,
    \item[(iii)] $L'' \sim L''_0+D_1+ D_2 +D_3$, where the $D_i$ are disjoint $(-1)$-curves and $L''_0$ is big and nef,
    \item[(iv)] there are two additional  $(-1)$-curves $D_4,D_5$ on $\widetilde{P}$,
mutually disjoint and disjoint from $D_1,D_2,D_3$,
satisfying $L''\cdot D_4 \neq L'' \cdot D_5$.
    \end{itemize}

    Then there exist $\overline {R}$ (respectively, $\overline {P}$) a blow--up of $R$ (resp., $P$) at $4$ (resp., $5$) points of $T$ and a line bundle $\overline L'$  on $\overline R$ (resp. $\overline L''$ on $\overline P$) such that:
    \begin{itemize}
    \item [(a)] the pair $(\overline R,\overline L')$ (resp. $(\overline P,\overline L'')$) is a deformation of $(\widetilde R, L')$ (resp. $(\widetilde P, L'')$),
    \item [(b)] the surface $\overline R\cup_T\overline P$ is a general member of $\mathcal D^*_{[4]}$ and $(\overline L', \overline L'')$ is a line bundle on it,
    \item [(c)] $(\overline R\cup_T\overline P, (\overline L', \overline L''))$ deforms to a smooth polarized Enriques surface $(S,L)$ such that $|L|$ contains an irreducible, rigid nodal elliptic curve.
    \end{itemize}
\end{prop}

\begin{proof}
  We argue as in the beginning of the previous proof with $k=3$, $s=4$ and $t=5$, noting that $L''$ is positive on at least one of the two curves $D_4,D_5$. We find, as before, a Cartier divisor in a surface
  \begin{equation} \label{eq:yinx}
    Y:=C \cup_T D \subset \widetilde{R} \cup_T \widetilde{P}=:X,\quad \widetilde{R}= \Bl_{y_1,y_2,y_3,y_4}(R), \quad  \widetilde{P}=\Bl_{y_5,\ldots,y_9}(P),
\end{equation}
for general $y_i \in T$, $i \in \{1,\ldots,9\}$, such that, denoting by $\e_i$ the exceptional divisor over $y_i$, one has
$$
C=C_0+\e_1+\e_2+\e_3, \quad D=D_0+\e_5+\e_6+\e_7,
$$
with $C\equiv L'$ and $D\sim L''$ both nodal, $C_0\in V^{**}_{1,m}(\widetilde R, T, L'_0)$, $D_0\in V^*_{0,m}(\widetilde P, T, L''_0)$, and we set $C_i=\e_i$ for
$i \in \{1,2,3\}$ and $D_i=\e_{i+4}$ for $i \in \{1,2,3,4,5\}$. Moreover, setting
$m:=C_0 \cdot T-3=D_0 \cdot T-3$,
there is a point $x\in T$ such that  
$$
C_0 \cap T  =  mx+y_5+y_6+y_7, \quad  D_0 \cap T  =  mx+y_1+y_2+y_3.
$$
As in the previous proof, the points $x,y_i$ are general on $T$.

The problem now is that we cannot a priori guarantee that $\sum_{i=1}^9y_i \in |\N_{T/{R}} \* \N_{T/{P}}|$ to ensure that $X$ is a member of $\D^*$ and therefore conclude as in the previous proof; we only know that $X$ is a member of $\D$. We will prove that we can create such a member of $\D^*$ without losing the ``nice'' properties of $Y$. 

Note first that condition (iv) says that $L'' \cdot \e_8 \neq L'' \cdot \e_9$. If $L''\cdot \e_i=0$ for $i =8$ or $9$, we may contract $\e_i$ and reduce to the case studied in the previous proposition. We will 
therefore assume that
\begin{equation} \label{eq:posi1}
  L''\cdot \e_8>0, \; \; L''\cdot \e_9>0\; \;  \mbox{and} \; \;  L'' \cdot \e_8 \neq L'' \cdot \e_9.
\end{equation}

    Fix now  general $y_1,y_2,y_3,y_5,y_6,y_7 \in T$. Varying $(x_4,x_8,x_9)$ in $T^3$ we obtain a 3--dimensional family of surfaces of the form 
    $$\Bl_{y_1,y_2,y_3,x_4}(R) \cup_T \Bl_{y_5,y_6,y_7,x_8,x_9}(P)$$
together with a family of lines bundles $(L',L'')$.
    There exists a relative Hilbert scheme $\mathcal H$ of effective
    Cartier divisors $C'\cup_TD'$ on these surfaces, such that on each
    surface the line bundle $(C',D')$ is  in the numerical equivalence
    class $[L',L'']$, cf.\ Remark \ref {rem:ingcro10} (considering
    $L'$ and $L''$ as relative line bundles, see subsection
    \ref{s:logS-family}).
    By our assumptions we have a nonempty subscheme $\W$ of $\mathcal H$  with a dominating morphism $\W \to T^3$ whose fiber over a point $(x_4,x_8,x_9)$ consists of pairs $(X',Y')$  such that
    \begin{itemize}
      \item $X'=R'\cup_TP'$ with $R'=\Bl_{y_1,y_2,y_3,x_4}(R)$ and $P' =\Bl_{y_5,y_6,y_7,x_8,x_9}(P)$,
      \item $Y'=C' \cup_T D'$, where
\begin{eqnarray*}
    C' & = & C'_0+\e_1+\e_2+\e_3\equiv L' \\
   D' & = & D'_0+\e_5+\e_6+\e_7\sim L'', \\
 &\mbox{with} & \hspace{-0.25cm} \mbox{$C'_0\in V^{**}_{1,m}(R',T,L'_0)$,  $D'_0\in V^*_{0,m}(P',T,L''_0)$ intersecting all $\e_i$ transversely,} \\
    C'_0 \cap T & = & mx'+y_5+y_6+y_7, \\
   D'_0 \cap T & = & mx'+y_1+y_2+y_3,\\
&\mbox{for} & \hspace{-0.3cm} \mbox{a point $x' \in T\setminus \{y_1,y_2,y_3, y_5,y_6,y_7\}$.} 
\end{eqnarray*}
    \end{itemize}
    We once and for all substitute $\W$ with a dominating component containing
    $[(X,Y)]$
and will henceforth assume $\W$ is irreducible. Taking the closure in $\mathcal H$, we obtain a closed scheme with a surjective morphism
    \[ g:\overline{\W} \to T^3. \] 
    If $[(X',Y')]\in \overline{\W}$, then $Y'$ looks like $C' \cup_T D'$ as above, except that
    $C'_0\in \overline{V}_{1,m}(R',T,L'_0)$ and $D'_0\in
    \overline{V}_{0,m}(P',T,L''_0)$, the intersection with the $\e_i$s need not be transversal and $x'$ may coincide with one of the
    points $y_i$s
(this follows from Lemmas \ref{lemma:R3}(i) and \ref{lemma:P3}(i)).  In any
event $(X',Y')$ comes equipped with a point $x' \in T$, which is the
only point of intersection between the non-exceptional members of
$Y'$. We will call this the {\it tacnodal point} of $[(X',Y')]$
(although it may be a worse singularity of $Y'$ for special pairs). We
therefore have a natural map
    \[ p: \overline{\W} \to T \] sending a pair to its tacnodal
point. For $x' \in T$ we set $\overline{\W}_{x'}:=p^{-1}(x')$; this is
the locus of pairs $[(X',Y')]$ of $\overline{\W}$ with tacnodal point
$x'$.

\begin{claim} \label{cl:10} The map $g$ is finite and $\dim
(\overline{\W})=3$.
 \end{claim}
    
\begin{proof}[Proof of claim] Since $g$ is surjective, we only need to
prove that $g$ is finite. Fix any $X'=R' \cup_T P'$ as above, with
$R'=\Bl_{y_1,y_2,y_3,x_4}(R)$ and $P' =\Bl_{y_5,y_6,y_7,x_8,x_9}(P)$,
and assume $[(X',Y')] \in \overline{\W}$. Let $\overline{C_0} \subset
R'$ and $\overline{D_0} \subset P'$ be the non-exceptional irreducible
components of $Y'$, elliptic and rational, respectively. They
intersect only at the tacnodal point $x' \in T$, and otherwise they
intersect $T$ in fixed points, as the $(-1)$-curves are fixed on each
component of $X'$. Therefore, by Lemma \ref {lemma:P3}(ii), there are
finitely many possibilities for the point $x'$ and the curve
$\overline{D_0}$, and consequently finitely many possibilities for the
intersection $\overline{C_0} \cap T$. As $\overline{C_0}$ is odd by
assumption (ii), there are by Lemma \ref {lemma:R3}(ii) only finitely
many possibilities for $\overline{C_0}$ as well. This proves that $g$
is finite. \end{proof}

    Let $\left(T^3\right)^*$ denote the two-dimensional subset of
triples $(x_4,x_8,x_9) \in T^3$ such that
\begin{equation}
  \label{eq:cond0} y_1+y_2+y_3+x_4+y_5+y_6+y_7+x_8+x_9 \in |\N_{T/R}
\* \N_{T/P}|,
\end{equation} and let $\overline{\W}^*=g^{-1}(\left(T^3\right)^*)$.
This is the locus of pairs $[(X',Y')]$ of $\overline{\W}$ such that
$X'$ is semi-stable. For $x' \in T$ we set
$\overline{\W}^*_{x'}:=\overline{\W}^*\cap \overline{\W}_{x'}$.

We write $g_1:\overline{\W} \to T$ and $g_2:\overline{\W} \to T \x T$
for the composition of $g$ with the projections onto the first factor
and onto the product of the second and third factors, respectively. In
other words $g_1$ maps a pair $(X'=R' \cup_T P',Y')$ as above to
$x_4$, whereas $g_2$ maps it to $(x_8,x_9)$.

\begin{claim} \label{cl:2} For all $x'\in T$ the following hold:
  \begin{itemize}
\item[(i)] $\overline{\W}^*_{x'} \neq \emptyset$ (hence $\overline
{\W}_{x'}\neq \emptyset$),
  \item[(ii)] $\left(g_1\right)|_{\overline{\W}_{x'}^*}$ is surjective
(hence also $\left(g_1\right)|_{\overline{\W}_{x'}}$ is surjective).
     \end{itemize}
\end{claim}

\begin{proof}[Proof of claim] Fix any $x_4\in T$
 
and consider $R'= \Bl_{y_1,y_2,y_3,x_4}(R)$.
By Lemma \ref{lemma:R3}(ii), there are finitely many curves
$C'_0\in \overline{V}_{1,m}(R',T,L'_0)$, 
such that $C_0' \cap T = mx'+y_5+y_6+y_7$. 
On the other hand, by Lemma \ref {lemma:P3-family}(ii),
taking property \eqref{eq:posi1} into account,
and considering the linear series
\begin{equation}
\label{eq:g(x4)}
  \mathfrak{g}(x_4) := |\N_{T/R} \* \N_{T/P}(-
  y_1-y_2-y_3-x_4-y_5-y_6-y_7)|
\end{equation}
of type $g^1_2$ on $T$,
there exist
finitely many curves $D'_0\in \overline{V}_{0,m}(P',T,L''_0)$,
with $P'=\Bl_{y_5,y_6,y_7,x_8,x_9}(P)$ for some
$x_8+x_9 \in \mathfrak {g}(x_4)$,
such that $D_0''\cap T = mx'+y_1+y_2+y_3$.
Thus
\[
  \bigr(R'\cup_TP',
  (C'_0+\e_1+\e_2+\e_3) \cup_T (D'_0+\e_5+\e_6+\e_7)\bigl)
  \in \overline{\W}^*_{x'} \cap g_1 ^{-1}(x_4),
\]
which proves (i) and (ii)
(note in particular that the condition that
$x_8+x_9 \in \mathfrak {g}(x_4)$ is equivalent to \eqref{eq:cond0}).

\end{proof}

Note that Claim \ref {cl:2} also implies that $p$ is surjective, whence all fibers $\overline{\W}_{x'}$ are two--dimensional by Claim \ref{cl:10}.

Now consider the map
\begin{eqnarray*}
  \sigma: T \x T & \longrightarrow & \Sym^2(T), \\
          (x,y) & \mapsto & x+y
  \end{eqnarray*}
and recall that there is a fibration $u:\Sym^2(T) \to T$ with fibers
being the $g^1_2$s on $T$.

For all $x_4 \in T$, we consider the linear series
$\mathfrak{g}(x_4) \subset \Sym^2(T)$ defined in \eqref{eq:g(x4)}
above;
it is a $g^1_2$, hence a fiber of $u$.

  \begin{claim} For all $x'\in T$ one has 
    $u(\sigma(g_2(\overline{\W}_{x'})))=T$, that is, $\sigma(g_2(\overline{\W}_{x'}))$ is not a union of fibers of $u$. 
  \end{claim}

  \begin{proof}[Proof of claim]
    Suppose to the contrary that $\sigma(g_2(\overline{\W}_{x'}))$ is a union of fibers of $u$. Then, for general $x_4 \in T$ we would have $\sigma(g_2(\overline{\W}_{x'})) \cap \mathfrak{g}(x_4) = \emptyset$, contradicting  Claim \ref{cl:2}(ii). 
  \end{proof}

  Let now $x'\in T$ be general.
  Set $\W_{x'}=\overline{\W}_{x'} \cap \W$, 
which is nonempty, as 
$[(X,Y)] \in \W_x$. 
  Since $\overline{\W}_{x'}$ is a general fiber of $p$ and $\overline{\W}$ is irreducible, $\W_{x'}$ is dense in any component of $\overline{\W}_{x'}$. It follows that $g_2(\W_{x'})$ is dense in any component of $g_2(\overline{\W}_{x'})$. By the last claim, $\sigma(g_2(\W_{x'})) \cap \mathfrak{g}(x_4) \neq \emptyset$ for general $x_4 \in T$. Pick any $(x_8,x_9) \in g_2(\W_{x'}) \cap \sigma^{-1}(\mathfrak{g}(x_4))$.
  Then by definition there exists a pair $[(X^1,Y^1=C^1 \cup_T D^1)] \in \W_{x'}$ such that
  \begin{eqnarray}
   \nonumber  \label{eq:caz1}
    & X^1=\Bl_{y_1,y_2,y_3,x'_4}(R) \cup_T \Bl_{y_5,y_6,y_7,x_8,x_9}(P) \; \; \mbox{for some} \; \; x'_4 \in T,& \text{and} \\
   \nonumber \label{eq:caz2} & y_1+y_2+y_3+x_4+y_5+y_6+y_7+x_8+x_9 \in |\N_{T/R} \* \N_{T/P}|.& 
\end{eqnarray}
In particular, $D^1\subset \Bl_{y_5,y_6,y_7,x_8,x_9}(P)$ satisfies the conditions of Proposition \ref{prop:GS} and
\begin{equation}\label{eq:caz3}
  D^1=D^1_0 + \e_5+ \e_6+\e_7, \; \; D^1_0 \cap T=mx'+y_1+y_2+y_3.  
  \end{equation}

  As $x_4$ is general in $T$, one has  $x_4 \in g_1(\W_{x'})$ (since $\left(g_1\right)|_{\overline{\W}_{x'}}$ is surjective, by Claim \ref {cl:2}(ii)). Then by definition there exists a pair $[(X^2,Y^2=C^2 \cup_T D^2)] \in \W_{x'}$ such that
  \begin{equation*}
    \label{eq:caz4}
     X^2=\Bl_{y_1,y_2,y_3,x_4}(R) \cup_T \Bl_{y_5,y_6,y_7,x'_8,x'_9}(P) \; \; \mbox{for some} \; \; x'_8,x'_9 \in T.
\end{equation*}
In particular, $C^2 \subset \Bl_{y_1,y_2,y_3,x_4}(R)$ satisfies the conditions of Proposition \ref{prop:GS} and
\begin{equation}\label{eq:caz5}
  C^2=C^2_0 + \e_1+ \e_2+\e_3, \; \; C^2_0 \cap T=mx'+y_5+y_6+y_7.  
  \end{equation}
  Consider the pair
   \[ (\overline X=\Bl_{y_1,y_2,y_3,x_4}(R) \cup_T \Bl_{y_5,y_6,y_7,x_8,x_9}(P),\overline Y=C^2 \cup_T D^1). \]
   Recall that $y_1,y_2,y_3, y_5,y_6,y_7$ were chosen general to start with, and $x_4$ and $x'$ are also general by construction. Lemma \ref {lemma:P3-family}(ii)  implies that we may choose $x_8+x_9$ general in $\mathfrak g(x_4)$. It follows that $\overline X$ is a general member of $\mathcal D_{[4]}^*$, i.e., (b) holds. Properties \eqref{eq:caz3} and \eqref{eq:caz5} imply that the pair $(\overline X, \overline Y)$   satisfies the conditions of Proposition \ref{prop:GS} and therefore  (c) holds.  It is moreover clear that also (a) holds. 
 \end{proof}

\section{Isotropic $10$-sequences and simple isotropic decompositions} \label{sec:iso}

An important tool for identifying the various components of the moduli spaces of polarized Enriques surfaces is the decomposition of line bundles as sums of effective isotropic divisors.  In this section we will recall some notions and results from \cite{cdgk,kn-JMPA}.

\begin{defn} (\cite[p.~122]{cd})  \label{def:rseq}
  An {\em isotropic $10$-sequence} on an Enriques surface $S$  is a sequence of isotropic effective divisors $\{E_1, \ldots, E_{10}\}$  such that $E_i \cdot E_j=1$ for $i \neq j$.
\end{defn}

It is  well-known that any Enriques surface contains isotropic $10$-sequences. Note that we, contrary to \cite{cd}, require the divisors to be {\it effective}, which can always be arranged by changing signs. We also recall the following result:

\begin{lemma} {\rm (\cite[Lemma 3.4(a)]{cdgk},\cite[Cor. 2.5.5]{cd})} \label{lemma:ceraprima}
   Let $\{E_1,\ldots,E_{10}\}$ be an isotropic $10$-sequence. Then there exists a divisor $D$ on $S$ such that $D^2=10$  and
$3D \sim E_1+\cdots+E_{10}$. Furthermore, for any $i \neq j$, we have
\begin{equation} \label{eq:10-3}
 D \sim E_i+E_j+E_{i,j}, \; \; \mbox{with $E_{i,j}$ effective isotropic,} \; \; E_i \cdot E_{i,j}=E_j \cdot E_{i,j}=2,
\end{equation} 
and 
$E_k \cdot E_{i,j}=1 \; \; \mbox{for} \; \; k \neq i,j$. 
 Moreover,
$E_{i,j} \cdot E_{k,l}= \begin{cases} 1, \; \mbox{if} \; \{i,j\} \cap \{k,l\} \neq \emptyset, \\
2, \; \mbox{if} \; \{i,j\} \cap \{k,l\} = \emptyset. \end{cases}$
\end{lemma}

The next result yields a ``canonical'' way of decomposing line bundles:

\begin{proposition} {\rm (\cite[Thm. 5.7]{kn-JMPA})} \label{prop:sid}
  Let $L$ be an effective line bundle on an Enriques surface $S$ such that $L^2>0$. Then there are unique nonnegative integers $a_0,a_1,\ldots,a_7,a_9$, $a_{10}$, depending on $L$,  satisfying
  \begin{eqnarray}
  \label{eq:condcoff}
  & a_1\geq \cdots \geq a_7, & \; \; \mbox{and} \\
  \label{eq:condcoff'}
  & a_9+a_{10} \geq a_0 \geq a_9 \geq a_{10}&
\end{eqnarray}
such that $L$ can be written as 
\begin{equation} \label{eq:scrivoL}
     L \sim a_1E_1+\cdots +a_7E_7+a_9E_9+a_{10}E_{10}+a_0E_{9,10}+\varepsilon_L K_S,
   \end{equation}
 for an isotropic $10$-sequence $\{E_1,\ldots,E_{10}\}$ (depending on $L$) and
\begin{equation} \label{eq:defeps}
\varepsilon_L= \begin{cases} 0, & \mbox{if $L+K_S$ is not $2$-divisible in $\Pic (S)$,} \\
1, & \mbox{if $L+K_S$ is $2$-divisible in $\Pic (S)$.} 
\end{cases}
\end{equation}
\end{proposition}

\begin{remark}
  Although the coefficients $a_i$ are unique, the isotropic $10$-sequence in Proposition \ref{prop:sid} is not unique, not even up to numerical equivalence or permutation, and nor is the presentation \eqref{eq:scrivoL}. See \cite[Rem. 5.6]{kn-JMPA}.
\end{remark}

\begin{defn}(\cite[Def. 5.1, 5.8]{kn-JMPA}) \label{def:fund}
Let $L$ be any effective line bundle on an Enriques surface $S$ such that $L^2>0$. A decomposition of the form \eqref{eq:scrivoL} with coefficients satisfying  \eqref{eq:condcoff}, \eqref{eq:condcoff'} and \eqref{eq:defeps}  is called a {\em fundamental presentation of $L$}. The coefficients $a_i=a_i(L)$, $i \in\{0,1,\ldots,7,9,10\}$ and $\varepsilon_L$ appearing in any
  fundamental presentation  are called {\em fundamental coefficients of $L$} or {\it of $(S,L)$}. 
  \end{defn}

\begin{remark} \label{rem:eps}
By \cite[Lemma 4.8]{cdgk} (or \cite[Thm. 1.3(f) and
Prop. 5.5]{kn-JMPA}) a line bundle $L$ is $2$-divisible in $\Num (S)$
if and only if all $a_i=a_i(L)$ are even, $i \in\{0,1,\ldots,7,9,10\}$. In particular, by \eqref{eq:defeps} or \cite[Cor. 4.7]{cdgk}, the number $\varepsilon=\varepsilon_L$ satisfies
  \begin{equation} \label{eq:condeps}
    \varepsilon= \begin{cases} 0, & \mbox{if some $a_i$ is odd,} \\
0 \; \mbox{or} \; 1, & \mbox{if all $a_i$ are even.} 
\end{cases}
\end{equation}
This means that any $11$-tuple $(a_0,a_1,\ldots,a_7,a_9,a_{10},\varepsilon)$ occuring as fundamental coefficients satisfies the conditions \eqref{eq:condcoff}, \eqref{eq:condcoff'} and \eqref{eq:condeps}. Conversely, for any such $11$-tuple we can choose any isotropic $10$-sequence on any Enriques surface and write down the line bundle \eqref{eq:scrivoL} having this $11$-tuple as fundamental coefficients. 
\end{remark}

For any integer $g \geq 2$, let 
$\E_{g}$ denote the moduli
space of complex polarized 
Enriques surfaces $(S,L)$ 
of genus $g$, which is a quasi-projective variety by  \cite[Thm. 1.13]{Vie}. Its irreducible components are determined by the fundamental coefficients, by the following:

\begin{thm} {\rm (\cite[Thm. 5.9]{kn-JMPA})} \label{thm:fundcoef}
Given an irreducible component $\E$ of $\E_{g}$,  all pairs $(S,L)$ in $\E$ have the same  fundamental coefficients.  Different components correspond to different fundamental coefficients. \end{thm}

The following technical result will be useful for our purposes:

\begin{lemma} \label{lemma:palloso}
  Let $(S,L)$ be an element of $\E_{g} \setminus \E_g[2]$. Set $a_i=a_i(L)$. Then one of the following holds:

  \begin{itemize}
  \item[(i)] There are three distinct  $k,l,m \in \{1,\ldots,7\}$ such that $a_i+a_k+a_l+a_m$ is odd for $i=9$ or $10$.
\item[(ii)] $a_0>0$ is odd and all $a_i$ for $i \neq 0$ are even.
  \item[(iii)] $a_0>0$ and all $a_i$ for $i \neq 0$ are odd.
  \end{itemize}
\end{lemma}

\begin{proof}
  We first show that if $a_0=0$, then we end up in case (i). We have $a_9=a_{10}=0$ by condition \eqref{eq:condcoff'}. Moreover, by Remark \ref{rem:eps}, the set $I:=\{i \in \{1,\ldots,7\} \; | \; a_i \; \mbox{is odd}\}$ is nonempty.
  If $\sharp I \geq 3$, then for any distinct $k,l,m \in I$ we have that
  $a_i+a_k+a_l+a_m=a_k+a_l+a_m$ is odd for $i=9$ and $10$. If $\sharp I \leq 2$, we may pick $k \in I$ and two distinct $l,m \in \{1,\ldots,7\}\setminus I$; then again $a_i+a_k+a_l+a_m=a_k+a_l+a_m$ is odd for $i=9$ and $10$.

  We may therefore assume that $a_0>0$.

  Assume next that (i) does not hold; then we have
  \begin{equation}
    \label{eq:semprepari}
    a_i+a_k+a_l+a_m \; \; \mbox{is even for all} \; \; i \in \{9,10\} \; \; \mbox{and distinct } \;  k,l,m \in \{1,\ldots,7\}. 
  \end{equation}
  This clearly implies that $a_9$ and $a_{10}$ have the same parity.

  Assume that $a_9$ and $a_{10}$ are even. Then \eqref{eq:semprepari} implies that $a_k+a_l+a_m$ is even for all distinct $k,l,m \in \{1,\ldots,7\}$. Hence $a_i$ is even for all $i\in \{1,\ldots,7\}$. 
  Then $a_0$ is odd by Remark \ref{rem:eps} and we end up in case (ii).

  Assume that $a_9$ and $a_{10}$ are odd. Then \eqref{eq:semprepari} implies that $a_k+a_l+a_m$ is odd for all distinct $k,l,m \in \{1,\ldots,7\}$.  Hence $a_i$ is odd for all $i\in \{1,\ldots,7\}$, yielding case (iii).
\end{proof}

\section{Isotropic $10$-sequences on members of $\D$} \label{sec:speciso}
The notions of isotropic divisors and isotropic $10$-sequences can be extended in the obvious way to all members of $\D$.
Referring to \cite[\S 3]{kn-JMPA} for more details, we will in Example \ref{ex:isoII} below construct one such 10--sequence that we will use in the proof of Theorem \ref {thm:main} in  the next section.

Recall that  we have the points  $y_1,\ldots,y_9 \in T$, which are the blown up points on either $R$ or $P$. We will now assume that $y_1,\ldots,y_9$ are distinct, though the case of coinciding points can be treated similarly. Denote by $\e_j$  the exceptional divisor over $y_j$, without fixing whether it lies on
$\widetilde{R}$ or $\widetilde{P}$. 

View $y_j \in T \subset P$. The linear system of lines in $P$ through $y_j$ is a pencil inducing a 
$g^1_2$ on $T$, 
 which has, by Riemann-Hurwitz, two members that  are also two fibers  of $\pi_{|T}: T \to E$. In other words, there are two fibers $\f_{\alpha_j}$ and $\f_{\alpha'_j}$ of $\pi:R \to E$ such that the intersection divisors
$\f_{\alpha_j} \cap T$ and $\f_{\alpha'_j}  \cap T$ belong to this $g^1_2$.  Since $\f_{\alpha_j}-\f_{\alpha_i}$ restricts trivially to $T$, one has 
$\alpha'_j=\alpha_j \+\eta$ (see Remark \ref {rem:ingcro10}). In particular, there are two uniquely defined points $\alpha_j$ and $\alpha_j \+ \eta$ on $E$ such that the pairs
\begin{eqnarray*}
  (\f_{\alpha_j}+\e_j,\ell) \; \; \mbox{and} \; \; (\f_{\alpha_j\+\eta}+\e_j,\ell),\; \; \mbox{if} \; \; \e_j \subset \widetilde{R}, \\
(\f_{\alpha_j},\ell-\e_j) \; \; \mbox{and} \; \; (\f_{\alpha_j\+\eta},\ell-\e_j),\; \; \mbox{if} \; \; \e_j \subset \widetilde{P},  
\end{eqnarray*}
define  distinct numerically equivalent  Cartier divisors on $X:=\widetilde{R} \cup_T \widetilde{P}$.  Hence   their difference is $K_{X}$   (see again Remark \ref {rem:ingcro10}).

Similarly, for four distinct (general) $y_i,y_j,y_k,y_l \in T$, the linear system of plane conics through $y_i,y_j,y_k,y_l$ is again a pencil inducing a 
$g^1_2$ on $T$. As above, there are  two fibers $\f_{\alpha_{ijkl}}$ and $\f_{\alpha_{ijkl} \+ \eta}$ of $\pi:R \to E$ such that the divisors 
$\f_{\alpha_{ijkl}} \cap T$ and $\f_{\alpha_{ijkl} \+\eta}  \cap T$
belong to this $g^1_2$.
In particular, the pairs
\begin{eqnarray*}
  (\f_{\alpha_{ijkl}}+\e_i+\e_j+\e_k+\e_l,2\ell) \; \; \mbox{and} \; \; (\f_{\alpha_{ijkl}\+\eta}+\e_i+\e_j+\e_k+\e_l,2\ell),\; \; \mbox{if} \; \; \e_i,\e_j,\e_k,\e_l \subset \widetilde{R}, \\
(\f_{\alpha_{ijkl}},2\ell-\e_i-\e_j-\e_k-\e_l) \; \; \mbox{and} \; \; (\f_{\alpha_{ijkl}\+\eta},2\ell-\e_i-\e_j-\e_k-\e_l),\; \; \mbox{if} \; \; \e_i,\e_j,\e_k,\e_l  \subset \widetilde{P},  
\end{eqnarray*}
together with similar pairs when $\e_i,\e_j,\e_k,\e_l$ are distributed  differently, define  Cartier divisors on $X$. One may again check that their difference is $K_{X}$.

Considering instead $y_i \in T \subset R=\Sym^2(E)$ we may write
$y_i=p_i+(p_i \+ \eta)$, for some $p_i \in E$. There are two sections in $R$ passing through $y_i$, namely
$\s_{p_i}$ and $\s_{p_i \+ \eta}$, cf.\ \eqref{eq:duesez}. Thus, the pairs
\begin{eqnarray*} 
  (\s_{p_i}-\e_i,0) \; \; \mbox{and} \; \; (\s_{p_i\+ \eta}-\e_i,0), \; \; \mbox{if} \; \; \e_i \subset \widetilde{R}, \\
 (\s_{p_i},\e_i) \; \; \mbox{and} \; \; (\s_{p_i\+ \eta},\e_i), \; \; \mbox{if} \; \; \e_i \subset \widetilde{P} 
\end{eqnarray*}
define  Cartier divisors on ${X}$. Again one may check that their difference is $K_{X}$.

\begin{example} \label{ex:isoII} 
  We consider $\widetilde{R}=\Bl_{y_1,y_2,y_3,y_4}(R)$ and $\widetilde{P}=\Bl_{y_5,\ldots,y_9}(\PP^2)$. Define
\begin{eqnarray*}
  E^0_i&:=& (\s_{p_i}-\e_i,0) \; \; \mbox{for} \; \; i \in\{1,2,3,4\},\\
  E^0_i & := & (\f_{\alpha_i},\ell-\e_i) \; \; \mbox{for} \; \; i \in\{5,6,7,8\}, \\
  E^0_9&:=& (\s_{p_9},\e_9), \\
  E^0_{10} &:= &(\f_{\alpha_{5678}}, 2\ell-\e_5-\e_6-\e_7-\e_8).
\end{eqnarray*}
These are all Cartier divisors on $X=\widetilde{R} \cup_T \widetilde{P}$ by the above considerations. One may  check that $(E^0_i)^2=0$ for all $i$ and $E^0_i \cdot E^0_j=1$ for all $i \neq j$. If $X$ is a member of $\D^*$, then, arguing as in the proof of \cite[Lemma 3.6]{kn-JMPA},
one may show that
\[
 E^0_1+\cdots+E^0_{10}-\xi \sim 3(E^0_9+E^0_{10}+E^0_{9,10}), 
\]
with $\xi$ as in \eqref{eq:xi} and
\[
  E^0_{9,10} =(\f_{\alpha_9},\ell-\e_9).
\]
Thus, we may similarly to \eqref{eq:10-3} define 
\[
E^0_{i,j}:=\frac{1}{3}\left(E^0_1+\cdots+E^0_{10}-\xi\right)-E^0_i-E^0_j \; \; \mbox{for each} \; \; i \neq j.
\]
In particular, letting $y_{78}\in T\subset \PP^2$ be the third intersection point of the line through $y_7$ and $y_8$ with $T$, and writing $y_{78}=p_{78}+(p_{78}\+ \eta)\in T\subset {\rm Sym}^2(E)$ for some $p_{78}\in E$, we will use that
\[
  E^0_{5,6} \sim (\s_{p_{78}},\ell-\e_7-\e_8).
\]
Note that $E^0_{9,10}$ and $E^0_{5,6}$ are Cartier divisors on any $X$ in $\D$.
\end{example}

\begin{remark} \label{rem:exdef}
  If $X=\widetilde{R} \cup_T \widetilde{P}$ belongs to $\D^*$, then, by Theorem
\ref{thm:deform}, as it deforms to a general Enriques surface $S$, the sequence $(E^0_1,\ldots,E^0_{10})$ deforms to an isotropic $10$-sequence $(E_1,\ldots,E_{10})$ on $S$ and each $E^0_{i,j}$ deforms to $E_{i,j}$ satisfying \eqref{eq:10-3}.
\end{remark}

\section{Proof of Theorem \ref{thm:main}} \label{sec:proofmain}

We are now ready to finish the proof of our main result Theorem \ref {thm:main}. Keeping in mind that the various irreducible components of $\E_g$ are determined by the fundamental coefficients of the line bundles they parametrize (cf.\ Theorem \ref{thm:fundcoef}), the proof will be divided in various cases depending on parity properties of the fundamental coefficients. Also recall that since we assume that we are not in $\E_g[2]$, at least one of the fundamental coefficients $a_i$ is odd and $\varepsilon=0$ (cf.\ Remark \ref{rem:eps}). In particular, the components we consider contain both $(S,L)$ and $(S,L+K_S)$, so there is no need to distinguish between linear and numerical equivalence classes (cf.\ also \cite[Thm. 1.1]{kn-JMPA}).

The proof strategy will be as follows: given fundamental coefficients $a_i$, find a suitable line bundle
with the same fundamental decomposition on some limit surface in terms of the isotropic divisors in Example \ref{ex:isoII}, and apply Proposition \ref{prop:metodo1} or  \ref{prop:metodo2} (and Remark \ref{rem:exdef}). As mentioned in the beginning of \S \ref{sec:prepar}, the existence of a rigid nodal elliptic curve will prove Theorem \ref{thm:main}.

We will first treat three special cases in \S \ref{ss:genere2}-\ref{ss:alpiu6} and then  the three cases of  Lemma \ref{lemma:palloso} in \S \ref{ss:case-i}-\ref{ss:case-iii}. 

\subsection{The case $a_1=a_2=1$ and $a_i=0$ otherwise} \label{ss:genere2}
Consider the limit line bundle  $L^0= E^0_8+E^0_9$, with $E^0_8$ and $E^0_9$
as in Example \ref{ex:isoII} (where we consider a general surface $\widetilde R\cup_T\widetilde P$ in $\mathcal D^*$).  Note that the order in an isotropic 10--sequence does not matter, hence we can choose $E^0_8$ and $E^0_9$ instead of $E^0_1$ and $E^0_2$. 
This fact will be used throughout the rest of the proof, without further mention.  Then
\[ L^0|_{\widetilde{R}} \equiv \f+\s \; \; \mbox{and} \; \; L^0|_{\widetilde{P}} \sim (\ell-\e_8)+\e_9.\]
In this case there is no need to invoke Proposition \ref{prop:metodo1} or  \ref{prop:metodo2}: indeed, the linear system $|L|$ contains the following curve:

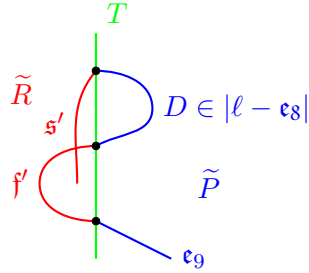
\begin{center}
\begin{tikzpicture}[scale=0.5]

\draw[green,thick] (0,-1) -- (0,5) node[above right] {$T$};

\draw[smooth,red,thick]
(0,0) to[out=180,in=-90] (-1.5,1) node[left] {$\mathfrak{f}'$} to[out=90,in=180]  (0,2);

\draw[smooth,red,thick]
(0,4) to[out=230,in=90] node[left] {$\mathfrak{s}'$}(-0.5,1);

\draw[smooth,blue,thick]
(0,0) -- (2,-1) node[right] {$\mathfrak{e}_9$};

\draw[smooth,blue,thick]
(0,2) to[out=30,in=-90] (1.5,3) node[right] {$D \in |\ell-\mathfrak{e_8}|$} to[out=90,in=0]  (0,4); 

\filldraw (0,0) circle (0.1);
\filldraw (0,2) circle (0.1);
\filldraw (0,4) circle (0.1);

\draw[red] (-2,3.5) node {$\widetilde{R}$};
\draw[blue] (3,1) node {$\widetilde{P}$};

\end{tikzpicture}
\end{center}
Here $\f'$ is the unique fiber passing through the point $y_9=\e_9\cap T$ and $D$ is the unique element of $|\ell-\e_8|$ passing through the point $y'_9$ such that $y_9+y'_9= T \cap \f'$; finally $\s'$ is one of the two sections $\s$ passing through the remaining intersection point of $D$ with $T$.

Arguing as in the proof of Proposition \ref{prop:GS}, this curve can be deformed to a one-nodal rigid elliptic curve of arithmetic genus $2$ in the linear system $|L|$ as $(\widetilde{R} \cup_T \widetilde{P},L^0)$ deforms to $(S,L)$. 

\subsection{The case $a_0=a_9$ and $a_i=0$ otherwise} \label{ss:genere3}

Both $a_0$ and $a_9$ are odd.

\subsubsection{Subcase $a_0=a_9=1$}

Consider the limit line bundle $L^0 = E^0_{9,10}+E^0_9$, with $E^0_{9,10}$ and $E^0_9$
as in Example \ref{ex:isoII} (as above). Then
\[ L^0|_{\widetilde{R}} \equiv \f +\s \; \; \mbox{and} \; \; L^0|_{\widetilde{P}} \sim (\ell-\e_9)+\e_9.\]
There is again no need to invoke Proposition \ref{prop:metodo1} or  \ref{prop:metodo2}: indeed, the linear system $|L|$ contains the following curve, constructed as in the previous case:

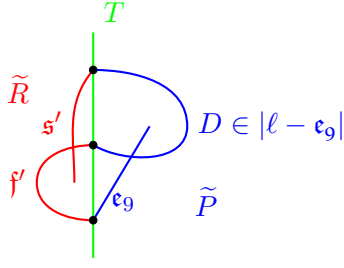
\begin{center}
\begin{tikzpicture}[scale=0.5]

\draw[green,thick] (0,-1) -- (0,5) node[above right] {$T$};

\draw[smooth,red,thick]
(0,0) to[out=180,in=-90] (-1.5,1) node[left] {$\mathfrak{f}'$} to[out=90,in=180]  (0,2);

\draw[smooth,red,thick]
(0,4) to[out=230,in=90] node[left] {$\mathfrak{s}'$}(-0.5,1);

\draw[smooth,blue,thick]
(0,0) --  (1.5,2.5);

\draw[smooth,blue,thick]
(0,2) to[out=-30,in=-90] (2.5,2.5) node[right] {$D \in |\ell-\mathfrak{e_9}|$} to[out=90,in=0]  (0,4); 

\filldraw (0,0) circle (0.1);
\filldraw (0,2) circle (0.1);
\filldraw (0,4) circle (0.1);

\draw[red] (-2,3.5) node {$\widetilde{R}$};
\draw[blue] (3,0.5) node {$\widetilde{P}$};
\draw[blue] (0.8,0.5) node {$\mathfrak{e}_9$};

\end{tikzpicture}
\end{center}
Arguing  again as in the proof of Proposition \ref{prop:GS}, this curve can be deformed to a rigid elliptic  two-nodal curve of arithmetic genus $3$ in the linear system $|L|$ as $(\widetilde{R} \cup_T \widetilde{P},L^0)$ deforms to $(S,L)$.

\subsubsection{Subcase $a_0=a_9 \geq 3$}

Consider the limit line bundle $L^0 = a_0(E^0_{5,6}+E^0_5)$, with $E^0_{5,6}$ and $E^0_5$
as in Example \ref{ex:isoII}. Then
\[ L':=L^0|_{\widetilde{R}} \equiv a_0\s +a_0\f \; \; \mbox{and} \; \; L'':=L^0|_{\widetilde{P}} \sim a_0(\ell-\e_7-\e_8)+a_0(\ell-\e_5)=a_0(2\ell-\e_5-\e_7-\e_8).\]

We see that we might as well substitute $\widetilde{R}$ with $R$ and $\widetilde{P}$ with $P'=\Bl_{y_5,y_7,y_8}(P)$ and consider $L^0=(L',L'')$ as a line bundle on $R \cup_T P'$. We apply Proposition \ref{prop:metodo1}
with $k=0$, $s=0$ and
$t=3$.

Clearly  $L'$ satisfies condition $(\star)$ and is odd and 
$L''$ is big and nef. Moreover, $\e_i \cdot L''>0$ for $i \in \{5,7,8\}$.
The conditions of Proposition \ref{prop:metodo1} are satisfied and we are done.

\subsection{The cases $a_7=a_9=a_{10}=a_0=0$} \label{ss:alpiu6}
At least one of the $a_i$ is odd. Pick the minimal such and call it $c_1$.  Reordering the remaining $a_i$s, we have that a  limit line bundle is of type
\[ L^0 \equiv c_1E^0_1+\sum_{i=5}^8c_iE^0_i+c_{10}E^0_{10}, \]
with the $E^0_i$ as in Example \ref{ex:isoII}, where $c_1$ is the minimal odd coefficient and we may also assume that $c_5>0$. Then
\begin{eqnarray*}
  L':=L^0|_{\widetilde{R}} & \equiv & c_1(\s-\e_1)+\sum_{i=5}^8c_i\f+c_{10}\f=c_1\s+(c_5+c_6+c_7+c_8+c_{10})\f-c_1\e_1, \\
  L'':=L^0|_{\widetilde{P}} & \sim & \sum_{i=5}^8c_i(\ell-\e_i)+c_{10}(2\ell-\e_5-\e_6-\e_7-\e_8). 
\end{eqnarray*}
We see that we might as well substitute $\widetilde{R}$ with $R'=\Bl_{y_1}(R)$ and $\widetilde{P}$ with $P'=\Bl_{y_5,y_6,y_7,y_8}(P)$ and consider $L^0=(L',L'')$ as a line bundle on $R' \cup_T P'$. We apply Proposition \ref{prop:metodo1}
with $s=1$, $t=4$ and $k=0$. Conditions (i)--(iii) of $(\star)$ are verified by $L'$; condition (iv) is equivalent to $c_5+c_6+c_7+c_8+c_{10}\geq 2$, which is verified unless $c_5=1$ and $c_6=c_7=c_8=c_{10}=0$. Since $c_1$ was assumed to be a minimal odd fundamental coefficient, we must have $c_1=1$ as well. This case is the one treated in \S \ref{ss:genere2}. We may therefore assume that $L'$ satisfies $(\star)$. 
One readily checks that $L'$ is odd (since $c_1$ is odd) and that $L''$ is big and nef (all components have square zero, intersect and can be represented by irreducible curves). This verifies conditions (i)--(iv) in Proposition \ref{prop:metodo1}. Since
for instance $\e_5 \cdot L''=c_5+c_{10} \geq c_5>0$, also condition (v) therein is satisfied.  Hence we are done by Proposition \ref{prop:metodo1}.

\subsection{Case where there are three distinct  $k,l,m \in \{1,\ldots,7\}$ such that $a_i+a_k+a_l+a_m$ is odd for $i=9$ or $10$ (case (i) in Lemma \ref{lemma:palloso})} \label{ss:case-i}

Note that the cases among these with $a_0=0$ (whence also $a_9=a_{10}=0$) and $a_7=0$  fall into the cases treated in \S \ref{ss:alpiu6}. We can therefore assume that
\begin{equation}
  \label{eq:alpiu6}
  a_7 >0 \; \; \mbox{(whence $a_i >0$ for all $i \in \{1,\ldots,7\}$), if $a_0=0$.}
\end{equation}
Similarly, the cases among these with $a_0=a_9$ and all remaining $a_i=0$ 
 fall into the cases treated in \S \ref{ss:genere3}. We can therefore assume that
\begin{equation}
  \label{eq:alpiu2}
  a_0 \neq a_9, \; \; \mbox{if} \; \; a_i =0 \; \; \mbox{for all}
  \; \; i \in \{1,\ldots,7,10\}.
\end{equation}

 We know that for $i=9$ or $10$, we can find indices $k,l,m$ so that $a_i+ a_k+a_l+a_m$ is odd. In this case we take $k,l,m$ so that $a_k+a_l+a_m$ is minimal with respect to this property. If we can do this for both  $i=9$ and $10$ (with possibly different triples of indices  $k,l,m$),  we will pick $i \in \{9,10\}$ so that $a_i+ a_k+a_l+a_m$ is minimal.
 We rename these coefficients $a_i,a_k,a_l,a_m$ as $c_9,c_2,c_3,c_4$,  making sure that 
 \begin{equation}\label{eq:c2} 
c_2 \geq c_3 \geq c_4,
\end{equation} 
set $c_0=a_0$,
 $c_{10}=\begin{cases} a_{10}, & \mbox{if} \; i=9 \\
  a_{9}, & \mbox{if} \; i=10
\end{cases}$,
and rename the remaining $a_i$ as $c_5,c_6,c_7,c_8$  in such a way that 
\begin{equation} \label{eq:c3}
 c_5 \geq c_6 \geq c_7 \geq c_8.
\end{equation} 
We thus have a limit line bundle
\[ L^0 \equiv c_0E_{9,10}^0+c_9E^0_9+c_{10}E^0_{10}+\sum_{i=2}^8c_iE^0_i, \]
with the $E^0_i$ and $E^0_{9,10}$ as in Example \ref{ex:isoII}, where,  besides \eqref {eq:c2} and \eqref {eq:c3}, one has 
\begin{eqnarray}
  \label{eq:c1}
  & c_9+c_{10} \geq c_0 \geq \max\{c_9,c_{10}\}, & \\
\label{eq:c4}& c_9+c_2+c_3+c_4 \; \; \mbox{is odd}, \\
 & \label{eq:c5} \mbox{there are no $i \in \{9,10\}$, $k,l,m \in \{2,\ldots, 8\}$ such that} \\
  \nonumber & \; \; \; \; \; \; \; \mbox{$c_i+c_k+c_l+c_m$ is odd and} & \\
 \nonumber  & \begin{cases}  c_k+c_l+c_m < c_2+c_3+c_4, \mbox{or} \\
      c_k+c_l+c_m = c_2+c_3+c_4 \; \; \mbox{and} \; \; c_i+ c_k+c_l+c_m < c_9 + c_2+c_3+c_4.
\end{cases} &
\end{eqnarray}
Furthermore, \eqref{eq:alpiu6} gives
\begin{equation}
  \label{eq:alpiu6'}
  c_i>0 \; \; \mbox{for all} \; \; i \in \{2,\ldots,8\}, \; \; \mbox{if} \; \; c_0=0,
\end{equation}
and \eqref{eq:alpiu2} yields
\begin{equation}
  \label{eq:alpiu2'}
c_0 \neq c_9 \; \; \mbox{if} \; \; c_i=0 \; \; \mbox{for all} \; \; i\in \{2,\ldots,8,10\}.
 \end{equation}

We define
\begin{eqnarray*}
  \kappa  :=  \sharp\{j \in \{2,3,4\} \; | \; c_j>0\} & \mbox{and} &
  \lambda :=  \sharp\{j \in \{5,6,7,8\} \; | \; c_j>0\}.
\end{eqnarray*}

\begin{claim} \label{cl:case-i}
  The following hold:
  \begin{itemize}
  \item[(i)]   If $c_0=0$, then $(\kappa,\lambda)=(3,4)$.
  \item[(ii)] If $\lambda \leq 2$, then $\kappa \leq 1$; moreover, $\kappa=1$ implies $c_{10} \geq 2$.
    \item[(iii)]  If $(\kappa,\lambda,c_{10})=(0,0,0)$, then $c_0 \neq c_9$.   
  \end{itemize}
\end{claim}

\begin{proof}
  Property (i) follows from condition \eqref{eq:alpiu6'}.

  Next assume $\lambda \leq 2$, that is, $c_7=c_8=0$. Then
properties \eqref{eq:c4} and \eqref{eq:c5} yield that $c_3=c_4=0$, that is, $\kappa \leq 1$, as we now explain. 

Indeed, if $c_4$ is even and positive we have that $c_9+c_2+c_3+c_8$ is odd and $c_2+c_3+c_8<c_2+c_3+c_4$, contradicting \eqref {eq:c5}. Similarly for the cases where $c_i$ is even and positive with $i=2,3$. From this it follows that none among $c_2,c_3, c_4$ can be even  and positive. 

If $c_3$ and $c_4$ are odd, then $c_9+c_2+c_7+c_8$ is odd and $c_2+c_7+c_8<c_2+c_3+c_4$, contradicting \eqref {eq:c5}. Similarly if $c_2$ and $c_3$ or $c_2$ and $c_4$ are odd. Hence at most one among $c_2,c_3, c_4$ is odd. 

In conclusion at least two among $c_2,c_3, c_4$ must be zero, hence $c_3=c_4=0$ by \eqref {eq:c2}, as we claimed. 

If $\kappa=1$, we have $c_2>0$ and $c_9+c_2$ is odd by \eqref{eq:c4}. Condition \eqref{eq:c5} yields that $c_9$ is even (whence $c_2$ is odd), for otherwise $c_9=c_9+c_3+c_4+c_8$ would be odd with
  $0=c_3+c_4+c_8<c_2=c_2+c_3+c_4$. For the same reason, $c_{10}$ is even, and condition \eqref{eq:c5} yields that $c_{10} \geq c_9$. By \eqref{eq:c1} and the fact that $c_0>0$ from (i), we must have $c_{10}>0$.
This proves (ii).

Finally, (iii) is a reformulation of property \eqref{eq:alpiu2'}. 
\end{proof}

Consider
\begin{eqnarray*}
  L' :=  L^0|_{\widetilde{R}} & \equiv & c_0\f+c_9\s+c_{10}\f +\sum_{i=2}^4 c_i(\s-\e_i)+\sum_{i=5}^8c_i\f \\
  & = & (c_2+c_3+c_4+c_9)\s+(c_0+c_5+c_6+c_7+c_8+c_{10})\f - \sum_{i=2}^4 c_i\e_i\\
  & = & L'_0 + \sum_{i=2}^{\kappa+1}(\f-\e_i),
\end{eqnarray*}
where
\[
  L'_0 :=(c_2+c_3+c_4+c_9)\s+(c_0+c_5+c_6+c_7+c_8+c_{10}-\kappa)\f-\sum_{i=2}^{\kappa+1} (c_i-1)\e_i
\]
and $\sum_{i=2}^{\kappa+1}(\f-\e_i)$ is the sum of $\kappa$ disjoint $(-1)$-curves.
We note that we may consider $L'$ as a line bundle on the blow--up of $R$ at $\kappa$ points. Hence we will eventually apply Proposition \ref{prop:metodo1} with $k=s=\kappa$.

\begin{claim}
  $L'_0$ verifies condition $(\star)$ and is odd.
\end{claim}

\begin{proof}
  Oddness is equivalent to condition \eqref{eq:c4}.
  Conditions (i)--(iii) of $(\star)$ are easily checked. Condition (iv) is equivalent to
  \begin{equation} \label{eq:cazziv}
    2c_0+c_9+2c_{10}+2\sum_{i=5}^8c_i -\kappa  \geq 4.
    \end{equation}
    If $c_0=0$, then $c_9=c_{10}=0$ by \eqref{eq:c1} and
    $(\kappa,\lambda)=(3,4)$ by Claim \ref{cl:case-i}(i), whence the left hand side of \eqref{eq:cazziv} equals $2\sum_{i=5}^8c_i -\kappa \geq 8 -3=5$, and we are done. Hence we may assume that $c_0>0$ for the rest of the proof.

    We note that, by \eqref{eq:c1},
    \[  2c_0+c_9+2c_{10}+2\sum_{i=5}^8c_i -\kappa \geq 3c_0+c_{10}+2\lambda-\kappa
      \geq 3+c_{10}+2\lambda-\kappa.\]
    This, together with Claim \ref{cl:case-i}(ii) tells us that \eqref{eq:cazziv} is always satisfied if $\lambda \geq 1$. Assume therefore that $\lambda=0$. Then  $\kappa \leq 1$ by Claim \ref{cl:case-i}(ii). If  $\kappa=1$, then $c_{10} \geq 2$ by Claim \ref{cl:case-i}(ii), and  \eqref{eq:cazziv} is again satisfied.
    If $\kappa=0$, we have that $c_9$ is odd by \eqref{eq:c4}. If $c_{10}>0$,  \eqref{eq:cazziv} is  satisfied. Otherwise  Claim \ref{cl:case-i}(iii) yields  $c_0 \geq 2$, whence \eqref{eq:cazziv} is again satisfied.
  \end{proof}

Consider
\begin{eqnarray*}
  L'' :=  L^0|_{\widetilde{P}} &  \sim & c_0(\ell-\e_9)+c_9\e_9+c_{10}(2\ell-\e_5-\e_6-\e_7-\e_8)+ \sum_{i=5}^8 c_i(\ell-\e_i) \\
  & = & (c_0-c_9)(\ell-\e_9)+c_9 \ell+ c_{10}(2\ell-\e_5-\e_6-\e_7-\e_8)+ \sum_{i=5}^8 c_i(\ell-\e_i). 
\end{eqnarray*}

The idea is now to apply Proposition \ref{prop:metodo1} with $k=\kappa$. 

\subsubsection{Subcase $\lambda=3,4$} We have $c_5 \geq c_6 \geq c_7>0$ by \eqref{eq:c3}. Define
\begin{eqnarray}
  \label{eq:L03}
  L_0''(3) &  := &  (c_0-c_9)(\ell-\e_9)+c_9 \ell+ c_{10}(2\ell-\e_5-\e_6-\e_7-\e_8)+   \\
 \nonumber  + \sum_{i=5}^7 (c_i-1)\hspace{-1.1cm} & & (\ell-\e_i) + c_8(\ell-\e_8) +(\ell-\e_5-\e_6)+(\ell-\e_6-\e_7)+(\ell-\e_7-\e_8), \\
 \nonumber    L_0''(2) & = & L_0''(3)+\e_6, \\
 \nonumber    L_0''(1) & = & L_0''(3)+\e_6+\e_7, \\
 \nonumber    L_0''(0) & = & L_0''(3)+\e_6+\e_7+\e_8.
  \end{eqnarray}
  Then  one may check that
\[ L'' = L_0''(\kappa)+
  \begin{cases} 0,    & \mbox{if} \; \; \kappa=0, \\
    \e_8, & \mbox{if} \; \; \kappa=1, \\
    \e_7+\e_8, & \mbox{if} \; \; \kappa=2, \\
  \e_6+\e_7+\e_8, & \mbox{if} \; \; \kappa=3. \\
\end{cases}
\]

\begin{claim} \label{cl:bigandnef}
  $L_0''(\kappa)$ is big and nef for all $\kappa \in \{0,1,2,3\}$.
\end{claim}

\begin{proof}
  Since $\e_i \cdot L_0''(3)>0$, for $i\in \{6,7,8\}$, it suffices to verify that
$L_0''(3)$ is big and nef. All divisors in the sum \eqref{eq:L03} are  nef, except for the last three,  which are irreducible. Nefness follows if the latter three intersect $L_0''(3)$ nonnegatively. We have
\begin{eqnarray*}
  L_0''(3) \cdot (\ell-\e_5-\e_6) & = & c_0+(c_7-1)+c_8-1+0+1 \geq 0, \\
  L_0''(3) \cdot (\ell-\e_7-\e_8) & = & c_0+(c_5-1)+(c_6-1)+1+0-1 \geq 0. 
\end{eqnarray*}
Finally
\[ L_0''(3) \cdot (\ell-\e_6-\e_7)  =  c_0+(c_5-1)+c_8+0-1+0 \geq c_0+c_8-1, \]
which is nonnegative, since by Claim \ref{cl:case-i}(i), either $c_0>0$, or $\lambda=4$ (whence $c_8>0$). This proves nefness. Bigness is easily checked.
\end{proof}

We apply Proposition \ref{prop:metodo1} with $k=s=\kappa$, $t=5$
and $L''_0=L_0(\kappa)$. What is left to be checked is condition (v). The set of additional $t-k=5-\kappa$ disjoint $(-1)$-curves on $\widetilde{P}$
is 
\[ 
  \begin{cases} \e_5,\e_6,\e_7,\e_8,\e_9,    & \mbox{if} \; \; \kappa=0, \\
    \e_5,\e_6,\e_7,\e_9, & \mbox{if} \; \; \kappa=1, \\
    \e_5,\e_6,\e_9 & \mbox{if} \; \; \kappa=2, \\
  \e_5,\e_9 & \mbox{if} \; \; \kappa=3, \\
\end{cases}
\]
and  $\e_5 \cdot L''(\kappa)  =  c_{10}+c_5 \geq c_5>0$, as $\lambda>0$, verifying condition (v) in Proposition \ref{prop:metodo1}.

\subsubsection{Subcase $\lambda \leq 2$}
We have $c_0>0$, $c_7=c_8=0$, and $\kappa \leq 1$ by Claim \ref{cl:case-i}(i)--(ii). 

If $\kappa =0$ we apply Proposition \ref{prop:metodo1} with $k=s=0$ and $t=5$.
Condition (v) therein is satisfied, as for instance $\e_5 \cdot L''=c_{10}+c_5$
and  $\e_9 \cdot L''=c_0-c_9$; indeed, if $\e_5 \cdot L''=0$, then $c_{10}=c_5=0$, whence $\lambda=0$, so Claim \ref{cl:case-i}(iii) yields $\e_9 \cdot L''>0$.

If $\kappa=1$, then $c_{10} \geq 2$ by Claim \ref{cl:case-i}(ii). Write $L''=L''_0+ \e_9$, with
\begin{eqnarray*} L''_0:=
  (c_0-c_9)(\ell-\e_9)+c_9 \ell+ (c_{10}-1)(2\ell-\e_5-\e_6-\e_7-\e_8)+ \\
  \sum_{i=5}^6 c_i(\ell-\e_i) +(2\ell-\e_5-\e_6-\e_7-\e_8-\e_9),
  \end{eqnarray*}
which is big and nef, since the only term with negative square is the last one, and one checks that $(2\ell-\e_5-\e_6-\e_7-\e_8-\e_9) \cdot L''_0=c_0+c_9+c_5+c_6-1 \geq 0$. We apply Proposition \ref{prop:metodo1} with $k=s=1$ and $t=5$.
Condition (v) therein is satisfied, as for instance $\e_5 \cdot L''=c_{10}+c_5 \geq c_{10} \geq 2$.

\subsection{Case where $a_0>0$ is odd and all remaining $a_i$ are even (case (ii) in Lemma \ref{lemma:palloso})} \label{ss:case-ii}

Since $a_9,a_{10}$ are even and $a_0$ is odd, we have
\[ a_9+a_{10}>a_0 >a_9 \geq a_{10},\]
which implies $a_0 \geq 3$ and $a_9,a_{10} \geq 2$. Rearranging indices,
we have a limit line bundle
\[ L^0 \equiv c_0E_{5,6}^0+\sum_{i=1}^8c_iE^0_i+c_{10}E^0_{10}, \]
with the $E^0_i$ and $E^0_{5,6}$ as in Example \ref{ex:isoII}, where
\begin{eqnarray}
  \label{eq:cc1}
  & c_5+c_{6} > c_0 > c_5 \geq c_6\geq 2, \; \; \mbox{$c_0$ is odd, $c_5,c_6$ are even}, & \\
  \label{eq:cc2} & c_4 \leq c_3 \leq c_2 \leq c_1 \leq c_7 \leq c_8 \leq c_{10}, \; \; \mbox{all even}. &
\end{eqnarray}
We define
\[
  \kappa  :=  \sharp\{j \in \{1,2,3,4\} \; | \; c_j>0\}.
\]

Consider
\begin{eqnarray*}
  L' :=  L^0|_{\widetilde{R}} & \equiv & c_0\s+\sum_{i=1}^4 c_i(\s-\e_i)+\sum_{i=5}^8c_i\f+c_{10}\f \\
                              & = & (c_0+c_1+c_2+c_3+c_4)\s+\left(c_5+c_6+c_7+c_8+c_{10}\right)\f - \sum_{i=1}^{\kappa} c_i\e_i \\
  & = & L'_0 + \sum_{i=2}^{\kappa}(\f-\e_i),
\end{eqnarray*}
where
\[
  L'_0 :=(c_0+c_1+c_2+c_3+c_4)\s+
  \left(c_5+c_6+c_7+c_8+c_{10}-\kappa+1\right)\f   -c_1\e_1 -\sum_{i=2}^{\kappa} (c_i-1)\e_i 
\]
and $\sum_{i=2}^{\kappa}(\f-\e_i)$ is the sum of $\max\{0,\kappa-1\}$ disjoint $(-1)$-curves.
We note that we may consider $L'$ as a line bundle on the blow--up of $R$ at $\kappa$ points. Hence we will eventually apply Proposition \ref{prop:metodo1} with $s=\kappa$ and $k=\max\{0,\kappa-1\}$.

\begin{claim}
  $L'_0$ verifies condition $(\star)$ and is odd.
\end{claim}

\begin{proof}
  Oddness follows since $c_0+c_1+c_2+c_3+c_4$ is odd by our assumptions \eqref{eq:cc1} and \eqref{eq:cc2}.
 Conditions (i)--(iii) of $(\star)$ are easily checked. Condition (iv) is equivalent to 
\[ c_0+2(c_5+c_6+c_7+c_8+c_{10}) \geq \kappa+2.\]
This is verified since, by \eqref{eq:cc1},
the left hand side is $\geq c_0+2c_5+2c_6 \geq 3+4+4=11$.
  \end{proof}

We have
\[
  L'' :=  L^0|_{\widetilde{P}}   \sim  c_0(\ell-\e_7-\e_8)+ \sum_{i=5}^8 c_i(\ell-\e_i) +c_{10}(2\ell-\e_5-\e_6-\e_7-\e_8).
 \]
We can view $L''$ as a line bundle on $\Bl_{y_5,y_6,y_7,y_8}(P)$.
The idea is now to apply Proposition \ref{prop:metodo1} with $k=\max\{0,\kappa-1\}$, $s=\kappa$ and $t=4$. 

\subsubsection{Subcase $c_7=0$}
By \eqref{eq:cc2} we have $c_1=c_2=c_3=c_4=0$, whence $\kappa=0$. We apply Proposition \ref{prop:metodo1} with $s=k=0$ and $t=4$. Condition (v) therein is satisfied, as for instance $\e_5 \cdot L''=c_5+c_{10} \geq c_{5} >0$ by \eqref{eq:cc1}.

\subsubsection{Subcase $c_7>0$}
By \eqref{eq:cc2} we have $c_7,c_8,c_{10} \geq 2$.

Define
\begin{eqnarray}
  \label{eq:sumcc}
  L_0''(3) :=  c_0(\ell-\e_7-\e_8)+ c_5(\ell-\e_5)
+\sum_{i=6}^8 (c_i-1)(\ell-\e_i) +&
  \\
 \nonumber    +c_{10}(2\ell-\e_5-\e_6-\e_7-\e_8) +(\ell-\e_6-\e_7)+(\ell-\e_7-\e_5)+
(\ell-\e_8-\e_6),&
  \end{eqnarray} 
  \begin{eqnarray*}
    L_0''(2) & = & L''(3)+\e_5, \\
    L_0''(1) & = & L''(3)+\e_5+\e_6, \\
    L_0''(0) & = & L''(3)+\e_5+\e_6+\e_7.
  \end{eqnarray*}
  Then one may check that, for $j \in \{0,1,2,3\}$:
\[ L'' = L_0''(j)+
  \begin{cases} 0,    & \mbox{if} \; \; j=0, \\
    \e_7, & \mbox{if} \; \; j=1, \\
    \e_6+\e_7, & \mbox{if} \; \; j=2, \\
  \e_5+\e_6+\e_7, & \mbox{if} \; \; j=3. \\
\end{cases}
\]

\begin{claim} \label{cl:bigandnef2}
  $L_0''(j)$ is big and nef for all $j \in \{0,1,2,3\}$.
\end{claim}

\begin{proof}
  Since $\e_i \cdot L_0''(3)>0$, for $i\in \{5,6,7\}$, it suffices to verify that
$L_0''(3)$ is big and nef. All divisors in the sum \eqref{eq:sumcc} are of nonnegative square, except for the first and the last three. We have, using \eqref{eq:cc1} and the fact that $c_5,c_6,c_7,c_8,c_{10} \geq 2$:
\begin{eqnarray*}
L_0''(3) \cdot (\ell-\e_7-\e_8) & = & -c_0+c_5+(c_6-1) \geq 0, \\
L_0''(3) \cdot (\ell-\e_6-\e_7) & = & c_5+(c_8-1)-1 \geq 2+1-1=2, \\
L_0''(3) \cdot (\ell-\e_7-\e_5) & = & (c_6-1)+(c_8-1)-1+1 \geq 1+1+0 \geq 2, \\
L_0''(3) \cdot (\ell-\e_8-\e_6) & = & c_5+(c_7-1)+1-1 \geq 2+1+0=3, 
  \end{eqnarray*}
which proves that $L_0''(3)$ is nef. It is easily verified that it is big. \end{proof}

 Now we apply Proposition \ref{prop:metodo1} with
 $k=\max\{0,\kappa-1\} \leq 3$, $s=\kappa$, $t=4$ and $L''_0=L_0(k)$. What is left to be checked is condition (v). This is satisfied because $\e_8 \cdot L_0''(k)=c_0+c_{8}-1+c_{10}+1=c_0+c_8+c_{10}>0$.

\subsection{Case where $a_0>0$ and all remaining $a_i$ are odd (case (iii) in Lemma \ref{lemma:palloso})} \label{ss:case-iii}

Rearranging indices,
we have a limit line bundle
\[ L^0 \equiv c_0E_{9,10}^0+c_9E^0_9+c_{10}E^0_{10}+\sum_{i=1}^7c_iE^0_i, \]
with the $E^0_i$ and $E_{9,10}^0$ as in Example \ref{ex:isoII}, where
\begin{eqnarray}
  \label{eq:ccc1}
  & c_9+c_{10} \geq c_0 \geq c_9 \geq c_{10} >0,  \; \; c_9,c_{10} \; \;\mbox{odd}, & \\
  \label{eq:ccc2} & 0<c_1 \leq c_2 \leq \cdots \leq c_6 \leq c_7, \; \; \mbox{all odd}.&
\end{eqnarray}

Consider
\begin{eqnarray*}
  L' & := & L^0|_{\widetilde{R}}  \equiv  c_0\f +c_9\s+c_{10}\f
            +\sum_{i=1}^4 c_i(\s-\e_i)+\sum_{i=5}^7c_i\f \\
     & = & (c_1+c_2+c_3+c_4+c_9)\s+(c_0+c_5+c_6+c_7+c_{10})\f-\sum_{i=1}^4 c_i\e_i  \\
& = & L'_0 + \sum_{i=2}^{4}(\f-\e_i),
\end{eqnarray*}
where
\[
  L'_0 := (c_1+c_2+c_3+c_4+c_9)\s +\left(c_0+c_5+c_6+c_7+c_{10}-3\right)\f-c_1\e_1-\sum_{i=2}^4 (c_i-1)\e_i
\]
and $\sum_{i=2}^{4}(\f-\e_i)$ is the sum of three disjoint $(-1)$-curves.
We will eventually apply Proposition \ref{prop:metodo2}.

\begin{claim}
 $L'_0$ verifies condition $(\star)$ and is odd.
\end{claim}

\begin{proof}
  Oddness follows since $c_9+c_1+c_2+c_3+c_4$ is odd by our assumptions \eqref{eq:ccc1}-\eqref{eq:ccc2}. Conditions (i)--(iv) of $(\star)$ easily follow from properties \eqref{eq:ccc1}--\eqref{eq:ccc2}.
\end{proof}

We have
\begin{eqnarray*}
  L'' :=  L^0|_{\widetilde{P}} &  \sim & c_0(\ell-\e_9)+c_9\e_9+c_{10}(2\ell-\e_5-\e_6-\e_7-\e_8)+ \sum_{i=5}^7 c_i(\ell-\e_i) \\
                               & = &
L_0''+\e_6+\e_7+\e_8,
\end{eqnarray*}
with
\begin{eqnarray*}
  L_0'' & := & (c_0-c_9)(\ell-\e_9)+c_9 \ell+ c_{10}(2\ell-\e_5-\e_6-\e_7-\e_8)+ \sum_{i=5}^7 (c_i-1)(\ell-\e_i)+\\
  & & +(\ell-\e_5-\e_6)+(\ell-\e_6-\e_7)+(\ell-\e_7-\e_8). 
\end{eqnarray*}

\begin{claim}
 $L''_0$ is big and nef.
\end{claim}

\begin{proof}
All terms in the expression of $L_0''$ right above have nonnegative square except for the last three. One computes
\begin{eqnarray*}
L_0'' \cdot (\ell-\e_5-\e_6) & = & c_0+(c_7-1)-1+0+1 \geq c_0 >0, \\
L_0'' \cdot (\ell-\e_6-\e_7) & = & c_0+(c_5-1)-1 \geq 0, \\
L_0'' \cdot (\ell-\e_7-\e_8) & = & c_0+(c_5-1)+(c_6-1)+1-1 \geq c_0>0, \\
\end{eqnarray*}
which shows that $L_0''$ is nef. One easily computes that it is big.
\end{proof}

We apply Proposition \ref{prop:metodo2}. What is left to be checked is condition (iv): The additional disjoint $(-1)$-curves on $\widetilde{P}$ are $\e_5$ and $\e_9$, and we have
\[
  \e_5 \cdot L_0''  =  c_{10}+c_5 >c_{10} \geq c_0-c_9=
      \e_9 \cdot L_0'',
\]
by \eqref{eq:ccc1} and \eqref{eq:ccc2}.

This concludes the proof of Theorem \ref{thm:main}.

 %%%%%%%%%%%%%%%%%%%%%%%%%%%%% (BIBLIOGRAPHY)%%%%%%%%%%%%%%%%%%%%%%%%%%%%%%%
%
%
%%%%%%%%%%%%%%%%%%%%%%%%%%%%%%%%%%%%%%%%%%%%%%%%%%%%%%%%%%%%%%%%%%%%%%%


\begin{thebibliography}{[CC]}



\bibitem {Be} A. Beauville, \emph{
Counting rational curves on K3
surfaces},  Duke Math. J. {\bf 97} (1999), 99--108.



\bibitem{BOPY} J. Bryan, G. Oberdieck, R. Pandharipande, Q. Yin, {\em Curve counting on abelian surfaces and threefolds}, Alg. Geom. {\bf 5} (2018), 398--463.

\bibitem{CH} L.~Caporaso, J.~Harris, {\it Counting plane curves of any genus}, Invent. Math. {\bf 131} (1998), 345--392. 



\bibitem{CaCi} F.~Catanese, C.~Ciliberto, {\it Symmetric products of elliptic curves and surfaces of general type with $p_g=q=1$}, J. Algebraic Geom. {\bf 2} (1993), 389--411.

\bibitem{chen} X.~Chen, {\it Rational curves on K3 surfaces}, J. Algebraic Geom. {\bf 8} (1999), 245--278. 

\bibitem{CGL} X.~Chen, F.~Gounelas, C.~Liedtke, {\it Rational curves
    on lattice-polarized K3 surfaces},
 Algebr. Geom. {\bf 9} (2022), no. 4, 443--475.


\bibitem{CC} L.~Chiantini, C.~Ciliberto, {\it On the Severi varieties of surfaces in $\PP^3$}, J. Algebraic Geom. {\bf 8} (1999), 67--83.

\bibitem{CS} L.~Chiantini, E.~Sernesi, {\it Nodal curves on surfaces of general type}, Math. Ann. {\bf 307} (1997), 41--56.
 
\bibitem{indam} C.~Ciliberto, T.~Dedieu, C.~Galati, A.~L.~Knutsen, {\it 
A note on Severi varieties of nodal curves on Enriques surfaces}, in "Birational Geometry and Moduli Spaces", Springer INdAM Series {\bf 39}  (2018), 29--36.  

\bibitem{cdgk} C.~Ciliberto, T.~Dedieu, C.~Galati, A.~L.~Knutsen, {\it  Irreducible unirational and uniruled components of moduli spaces of  polarized Enriques surfaces}, 
Math. Z.  {\bf 303}, Article number: 73 (2023), doi: 10.1007/s00209-023-03226-5.


\bibitem{cdgk2} C.~Ciliberto, T.~Dedieu, C.~Galati, A.~L.~Knutsen,
  {\it Severi varieties on blow--ups of the symmetric square of an
    elliptic curve},
 Math. Nachr. {\bf 296} (2023), 574--587. 

\bibitem{cm} C.~Ciliberto, R.~Miranda, {\it Homogeneous interpolation on ten points}, J. Algebraic Geom.  {\bf 20} (2011), 685--726.  

  



\bibitem{cd}  F.~R.~Cossec,  I.~V.~Dolgachev,  {\it Enriques surfaces. I.}
  Progress in Mathematics, {\bf 76}. Birkh{\"a}user Boston, Inc., Boston, MA, 1989.

\bibitem{De} O. Debarre, {\em On the Euler characteristic of generalized Kummer varieties}, Amer. J. Math. {\bf 121} (1999), 577--586.

\bibitem{noteTD} Th.~Dedieu, {\it Geometry of logarithmic Severi
    varieties at a general point},
  \url{https://hal.archives-ouvertes.fr/hal-02913705}.








 \bibitem{fri} R.~Friedman, {\it Global smoothings of varieties with normal crossings}, Annals of Math. {\bf 118} (1983), 75--114.

  \bibitem{fri2} R.~Friedman, {\it A new proof of the global Torelli theorem for $K3$ surfaces}, Annals of Math. {\bf 120} (1984), 237--269.


 \bibitem{gk} C.~Galati, A.~L.~Knutsen, {\it On the existence of curves with $A_k$-singularities on $K3$ surfaces}, Math. Res. Lett. {\bf 21} (2014),
   1069--1109.
   
    \bibitem{Go} L. G{\"o}ttsche, {\em A conjectural generating function for numbers of curves on surfaces}, Comm. Math. Phys. {\bf 196} (1998), 523--53.


 \bibitem{GLS} G.-M.~Greuel, C.~Lossen, E.~Shustin, {\it 
    Geometry of families of nodal curves on the blown-up projective plane},
    Trans. Amer. Math. Soc. {\bf 350} (1998), 251--274. 
    
    \bibitem{Ha} J.~Harris, {\it On the Severi problem}, Invent. Math. {\bf 84} (1986), 445--461.
    
    


\bibitem{KMPS} A. Klemm, D. Maulik, R. Pandharipande, E. Scheidegger, {\em Noether-Lefschetz theory and the Yau-Zaslow conjecture}, J. Amer. Math. Soc. {\bf 23} (2010), 1013--1040.


\bibitem{kn-JMPA} A.~L.~Knutsen, {\it On moduli spaces of polarized
    Enriques surfaces}, J. Math. Pures Appl. {\bf 144} (2020), 106--136. 


\bibitem{KL} A.~L.~Knutsen, M.~Lelli-Chiesa, {\it Genus two curves on
    abelian surfaces},
Ann. Sci. \'Ec. Norm. Sup\'er. (4) {\bf 55} (2022), no. 4, 905--918.


   \bibitem{KLM} A.~L.~Knutsen, M.~Lelli-Chiesa, G.~Mongardi, {\it Severi Varieties and Brill-Noether theory of curves on abelian surfaces,} J. Reine Angew. Math. {\bf 749} (2019), 161--200. 

\bibitem{LS} H.~Lange, E.~Sernesi, {\it Severi varieties and branch curves of
abelian surfaces of type $(1, 3)$,} International J. of Math. {\bf 13},
(2002), 227--244.



     
\bibitem{MM} S.~Mori, S.~Mukai, {\it The uniruledness of the moduli space of curves of genus 11}, Algebraic geometry (Tokyo/Kyoto, 1982), 334--353, Lecture Notes in Math., 1016, Springer, Berlin, 1983. 

\bibitem {PS} R. Pandharipande, J. Schmitt, \emph{Zero cycles  on the moduli space of curves},  {\'E}pijournal de G{\'e}om{\'e}trie Alg{\'e}brique, {\bf 4},  article 12, (2020).



  
  
  \bibitem{ser} E.~Sernesi, {\it Deformations of Algebraic Schemes},
  Grundlehren der mathematischen Wissenschaften {\bf 334}. Springer-Verlag, Berlin, Heidelberg, New York, 2006.
  
  \bibitem{sev} F.~Severi, {\it  Vorlesungen {\"u}ber Algebraische Geometrie}, Johnson Pub. (reprinted,
1968; 1st. ed., Leipzig 1921).

\bibitem{Ta} A.~Tannenbaum, {\it Families of algebraic curves with nodes},
Compositio Math. {\bf 41} (1980), 107--126. 

\bibitem{Ta2} A.~Tannenbaum, {\it Families of curves with nodes on K3 surfaces},
Math. Ann. {\bf 260} (1982), 239--253.


\bibitem{Te} D.~Testa, {\it The irreducibility of the moduli spaces of rational curves on del Pezzo surfaces}, J. Algebraic Geom. {\bf 18} (2009), 37--61.

\bibitem{Te2} D.~Testa, {\it The Severi problem for rational curves on del Pezzo surfaces}, Ph.D.-Thesis, Massachusetts Institute of Technology, 2005, ProQuest LLC. 



  

 \bibitem{Ty} I.~Tyomkin, {\it On Severi varieties on Hirzebruch surfaces}, Int. Math. Res. Not. IMRN 2007, no. 23, Art. ID rnm109, 31 pp. 
 
 \bibitem {Tz} Y.~J.~Tzeng, \emph{A proof of the G{\"o}ttsche-Yau-Zaslow formula}, 
	J.  Diff. Geom. {\bf 90} (2012), 
	439--472.



 \bibitem{YZ} S. T. Yau, E. Zaslow, \emph{BPS states, string duality, and nodal curves on K3}, Nuclear Physics B. {\bf 471} (1996), 503--512.

 \bibitem{Vie} E.~Viehweg, {\it Quasi-projective Moduli for Polarized Manifolds}, 
 Ergebnisse der Mathematik und ihrer Grenzgebiete. 3. Folge, Vol. {\bf 30} (1995), Springer-Verlag Berlin Heidelberg. 
 
 \bibitem{Za}  A. Zahariuc, {\it The Severi problem for abelian
     surfaces in the primitive case},
   J. Math. Pures Appl. (9) {\bf 158} (2022), 320--349.
\end{thebibliography}
\end{document}